\documentclass{article}
\usepackage{amsthm}
\usepackage{amsfonts}
\usepackage{amssymb}
\usepackage{amsmath}
\usepackage{graphicx}
\usepackage{fullpage}
\usepackage{setspace}
\usepackage{natbib}

\date{February 22, 2016}

\newtheorem{theorem}{Theorem}

\newtheorem{lemma}{Lemma}

\newtheorem{proposition}{Proposition}

\renewcommand{\bar}[1]{\overline{#1}}

\begin{document}

\title{Group Incentives and Rational Voting \thanks{Earlier versions of this paper have been presented at the Midwest Political Science Association meeting (April 2012), the W. Allen Wallis Institute of Political Economy Conference (September 2012) and Warwick University's Political Economy Conference (March 2013). We gratefully acknowledge the useful comments of the audiences at these meetings and others. We are particularly grateful to Avidit Acharya, Scott Ashworth, Ethan Bueno de Mesquita, David Myatt, John Morgan, Barry Nalebuff, Scott Tyson and Srinivasa Varadhan. The research reported here was done while Tom LaGatta was a Courant Instructor \& PIRE Postdoctoral Fellow at the Courant Institute of Mathematical Sciences at NYU and his research and travel were supported in part by NSF PIRE Grant OISE-0730136.}}
\author{Alastair Smith\thanks{Alastair.Smith@nyu.edu} \and Bruce Bueno de Mesquita \thanks{bruce.buenodemesquita@nyu.edu} \and Tom LaGatta\thanks{Courant Institute, NYU and Splunk, Inc. lagatta@cims.nyu.edu;  tlagatta@splunk.com} }

\maketitle

\begin{doublespacing}

\begin{abstract}
Our model describes competition between groups driven by the choices of self-interested voters within groups. Within a Poisson voting environment, parties observe aggregate support from groups and can allocate prizes or punishments to them. In a tournament style analysis, the model characterizes how contingent allocation of prizes based on relative levels of support affects equilibrium voting behavior. In addition to standard notions of pivotality, voters influence the distribution of prizes across groups.  Such prize pivotality supports positive voter turnout even in non-competitive electoral settings. The analysis shows that competition for a prize awarded to the most supportive group is only stable when two groups actively support a party.  However, competition among groups to avoid punishment is stable in environments with any number of groups. We conclude by examining implications for endogenous group formation and how politicians structure the allocation of rewards and punishments.
\end{abstract}

\newpage

\section{Introduction}

To attain and retain power in an electoral setting politicians need to motivate their supporters to turn out and vote. While simply offering more or better rewards is one means to elicit support, we contend that politicians can do more with fewer resources by offering to allocate benefits across groups in a contingent manner-- a mechanism we refer to as a Contingent Prize Allocation Rule \citep{smith2012contingent}. As a simple illustrative example, a politician might offer to build a park (the prize) in the precinct (the group) that provides her with the most votes. 
Most rational choice explanations of voting examine pivotality and the extent to which an individual's vote is likely to influence who wins the election. In contrast, we contend that voters can be pivotal on other dimensions \citep{schwartz1987your}; and in particular we focus on the extent to which an individual voter shapes the distribution of prizes and punishments across groups. Hence we provide a link between individual rational choices at one level and the importance of groups in shaping political outcomes at another.  

Our approach is akin to the tournaments approach of \citet{lazear&rosen1981}. They examine how firms set wage schedules to incentivize the effort workers make by awarding a wage bonus to the most productive worker. However, within the political setting, simple wage competitions are more difficult to structure, not least because individual votes are anonymous. Vote buying occurs and patronage-style parties attempt to undermine the secret ballot. However, monitoring and rewarding each individual voter is expensive, time consuming and, empirically, appears the exception rather than the rule \citep{stokes2007political}. Instead, here we examine a setting where politicians observe political support (in terms of vote totals) at the group level (such as precincts, wards, or districts). Supportive groups are disproportionately rewarded, or alternatively, non-supportive groups are punished. Analogous to Lazear and Rosen's wage bonus for the most productive worker, we examine the implications of winner-take-all schemes that allocate a prize to the most supportive group. In addition, we model how punishing the least supportive group shapes the incentives of individuals within groups. 
In our analysis, group competition takes a pre-eminent role in shaping political outcomes although, and importantly, the power of these groups is derived by the actions of individual voters and their self-interested motivations.

We are not concerned here with comparing the properties of all possible reward or punishment mechanisms. Rather, our interest is to establish that contingent rules can significantly incentivize voting. Elsewhere we investigate the effects of a broader range of rules on group effort in political competition
(Smith and Bueno de Mesquita 2015).\nocite{bdmsmith2015effort} Here, we model the way that a winner-take-all tournament system influences voter turnout and voter incentive to form groups. We will see that politicians can exploit those incentives to mobilize turnout even when they face no credible political opposition. 

Winner-take-all is a useful starting place. The tournament literature which we apply assumes a winner-take-all environment and anecdotal evidence supports the idea that political parties routinely use such a mechanism. Stories of snow removal, for instance, in New York City and in Chicago's Democratic wards abound indicating that the most supportive neighborhoods are privileged. Likewise, as we discuss later, American political parties formalize the winner-take-all prize mechanism in awarding participation in their national nominating conventions. As Grossman and Helpman (2001, p. 226) observe, reward mechanisms need not include an explicit quid-pro-quo. In their analysis of rent-seeking by special interest groups (SIG), they contend,``Influence can be bought and sold by a subtle exchange in which both sides recognize what is expected of them. The SIG can make known by words and deeds that it supports 
politicians who are sympathetic to its cause. Then the policy maker can appear to be taking actions to promote a constituentÕs interests while gratefully accepting the groupÕs support."  The underlying logic requires that voters believe that politicians recognize and reward supportive groups. \nocite{GH_SIG2001}

We model the impact of rewarding the most supportive group within the contexts of Myerson's (1998, 2000) Poisson voting games. \nocite{myerson:1998:ijgt,myerson:2000:jet} This framework assumes there is a relatively large electorate with ambiguity in the precise number of voters, which is treated as a Poisson random variable. In common with much of the literature on rational voting, Myerson focuses on the extent to which individual voters are pivotal in determining the outcome of the election. The pathologies associated with such approaches are that turnout is predicted to be low and elections are nearly always close \citep{green/shapiro:1994}.  Contingent prizes create additional incentives to vote beyond simply affecting who wins; voters are also instrumental in shaping the distribution of the prizes allocated by parties. Such `prize pivotality' motivates voters to turn out even when voting is costly and the outcome is anticipated to be lopsided.  Similarly, 'punishment pivotality' induces individuals within groups to turnout to avoid a group punishment. Therefore, we focus on both carrots and sticks. 

Prize pivotality provides an incentive to support a candidate. By carefully crafting the competition for prizes, a politician makes a voter more influential over the distribution of group-oriented prizes and punishments than the voter is over which candidate wins the election. Just as a wage bonus induces workers to be more productive, the competition between groups for the prize increases electoral support. Further, such a boost in the incentive to vote does not require that an election is close. Indeed, to isolate the impact of prize competition we initially examine many of our results in the setting of non-competitive elections. 

Rewarding supportive groups is a standard practice within the party machines that have dominated many large US cities (Allen 1993). Richard J. Daley, the long term mayor of Chicago, was notorious in this respect (Rakove 1976). This is perhaps unsurprising since the internal rules of the Democratic Party of Cook County (which contains Chicago) specifies that on committees, ward representatives are given voting rights in proportion to the level of democratic votes their ward delivered in previous elections (Gosnell 1937). As Rakove reports, the Democratic Party shifts which groups receive rewards in response to changes in their level of support:  
``The machine co-opts those emerging leaders in the black and Spanish-speaking communities who are willing to cooperate; reallocates perquisites and prerogatives to the blacks and the Spanish Speaking, taking them from ethnic groups such as the Jews and Germans, who do not support the machine as loyally as their fathers did (Rakove 1976. p.16)."
\nocite{allen1993,gosnell1937,rakove1976,tam2003,newyorkpost}
US national parties also structure rules to reward their loyalists. For instance, both parties skew representation in presidential nominating conventions in favor of states that gave the party high levels of support in previous elections (see for instance, Democratic Party Headquarters 2007). 
\nocite{dph2007}
In other systems, punishments are more prominent. For instance, the People's Action Party of Singapore is notorious for cutting public housing and services to neighborhoods that fail to support it in elections (Tam 2003). Penalties also occur in US cities. For instance, after heavy snowfall in January 2014, The New York Post reported under the headline ``De Blasio 'getting back at us' by not plowing" (January 21, 2014) that ``It really is a tale of two cities Ñ this time with the tony Upper East Side getting the shaft!
Huge swaths of the cityÕs wealthiest neighborhood had been not been plowed by early Tuesday evening, leaving 1-percenters out in the cold, according to the cityÕs own map of snow-plower activity." That is, the neighborhood that had given the mayor little support in his election somehow got overlooked when it came to clearing the snow! Under the previous mayor, Michael Bloomberg, who was supported by Upper East Side (UES) voters, the UES had  been one of the first neighborhood to get plowed. 

The formal analysis of equilibrium behavior predicts that high turnout competition over prizes is only stable between two groups but competition to avoid group-based punishments is stable for any number of groups. We examine the implications of this Duvergerian style result \citep{duverger1959political,riker1982two} in the rewards setting. In particular, we discuss the incentives engendered for individuals to migrate between groups; in essence, altering their group identity. Finally we examine how politicians can increase their electoral support by breaking the competition for prizes into a series of smaller tournaments. When prizes are non-rival; that is, each group member's utility from a prize is undiminished by additional group members, politicians should optimally structure competition between two large groups. In contrast, when prizes are rival in nature (such as a cash transfer to the group), a politician engenders greater support by creating a large number of competitions between pairs of groups. For instance, when deciding which neighborhoods to snowplow first, a wily politician should pair off precincts or neighborhoods and plow the supportive neighborhoods first.

\section{Literature Review}
We consider a tournament style competition in which politicians offer group based rewards contingent on the relative number of votes delivered by each group. In the basic formulation of tournaments a firms offers differential wages based on the rank order of worker productivity \citep{lazear&rosen1981,rosen1986prizes,becker1992incentive}. One standard interpretation is that the most productive worker receives a promotion (for reviews of the tournament literature see \citet{connelly2014tournament} and \citet{prendergast1999provision}). By offering a prize to the most productive worker, firms motivate worker effort. We exploit an analogous approach in which parties offer group-based prizes to the groups that deliver the most votes. Although collective action problems persist because the group prize is essentially a public good to all members of the group \citep{palfrey1984participation}, such prizes fuel political participation because voters can have greater influence over the distribution of prizes than they have over who wins the election. 

Given that any voter has a nearly zero chance of influencing the electoral outcome and voting involves some cost in time and effort, it is for many a puzzle why there is turnout. Several different modeling strategies have been suggested to account for the reality of relatively high turnout in mass elections . For instance, \citet{evren2012altruism} and \citet{feddersen2006theory} introduce altruistic voters. \citet{HerreraMorelliPalfrey2014} and \citet{kartal2014comparative} examine the impact of electoral rules on turnout and \citet{borgers:2004} contrasts the welfare implications for endogenous versus compulsory voting.  Pivotality plays a central role in virtually all rational choice models of voting \citep{aldrich:1993,downs:1957,ferejohn/fiorina:1974,ledyard:1984,palfrey/rosenthal:1983,palfrey/rosenthal:1985,riker/ordeshook:1968}. In the basic rational actor voting model an individual's vote matters only if it turns a loss into a draw or a draw into a victory. In a large electorate, even if the outcome is expected to be close,
the probability that a voter's vote matters is extremely small, leading to the claims of turnout pathology within the rational voter framework (\citealt{green/shapiro:1994}; see \citealt{feddersen:2004} and \citealt{geys:2006:psr} for surveys of this literature).  

To model pivotal events in the context of this apparent turnout pathology with the number of voters known requires the analyst to work within the context of the binomial distribution. As this proves to be technically demanding, alternative approaches have been suggested. \citet{myerson:1998:ijgt} suggests inducing uncertainty about the precise number of voters within the population and modeling the number of votes for each party as a Poisson random variable. This approach greatly simplifies combinatoric calculations and is adopted here. Others, \citet{krishna2011overcoming} for instance, similarly exploit this approach. Another alternative approach utilized by \citet{good-mayer:1975:bs}, \citet{krishna2012majority} and \citet{myatt2012rational} is to introduce aggregate uncertainty over parameters in the model and examine the ratio of limit pivot probabilities as the electorate becomes large (see also \citet{chamberlain1981note}, \citet{ acharya2015sincere} and \citet{mandler2012fragility}). In a recent working paper, \citet{myattsmith2014} introduce aggregate certainty to a similar model to the one examined here and characterize the ratio of the likelihood that a voter is pivotal in who wins the election and the likelihood that a voter is pivotal in the allocation of prizes. 

\citet{morton1991} and \citet{uhlaner1989rational} argue that group membership shapes turnout due to rewards provided within groups.  Linking such group based rewards to the pivotality arguments of our paper, \citet{shachar1999follow} find that local political party leaders and groups work harder to mobilize voters in US presidential elections when the state level result is predicted to be close. That is to say, group leaders try harder when the election will be close in their state. Such arguments reflect the decision theoretic arguments of \citet{schwartz1987your}. He argues voters can be pivotal on many dimensions and majority support for the victorious party is a motivating factor at the local level. 

Our approach reflects pivotality concerns for voters beyond the outcome of the election. Other scholars similarly argue that voting is about more than simply who wins the immediate election.  \citet{castanheira2003victory}, \citet{meirowitz2009pivots} and \citet{razin2003signaling} explore how election results shape candidates' issue positions in future elections. Reminiscent of such signaling ideas, in \citet{myatt2012theory} voters want to signal their dislike of certain policies through a protest vote and in doing so tradeoff the probability that they are pivotal in delivering sufficient protest with the risk that their vote is pivotal in allowing an opposition party to win. \citet{myatt:2007} examines strategic voting in which voters who want to depose the incumbent must balance their preferences over opposition parties with the electoral prospects of these parties. In \citet{dewan/myatt:2007} it is party leaders who must tradeoff their desire to support their preferred candidate with the need to present a unified policy position to the voters.  

\section{Model Setup}

We assume an election takes place between two parties, $\mathcal{A}$ and $%
\mathcal{B}$, for a single office. All voters have the option of voting for
party $\mathcal{A}$, voting for party $\mathcal{B}$ or abstaining. Each
voter pays a cost $c$ to vote; abstention is free.  To begin, we assume there is a large number of voters,
divided into $K$ roughly equally sized groups. The groups are indexed $%
1,2,\dots ,K$. Although these groups might be based on any underlying
societal cleavage, for convenience, we treat them as though they are
geographically based wards within an electoral district.

Group $k$ has size $N_{k}$ which we treat as an unknown Poisson random
variable with mean $n_{k}$. Therefore, $\Pr (N_{k}=x)=f_{n_{k}}(x)=\frac{%
n_{k}^{x}}{x!}e^{-n_{k}}$ and $\Pr (N_{K}\leq x)=F_{n_{k}}(x)=\frac{\Gamma
(x+1,n_{k})}{x!}=\sum_{z=0}^{x}f_{n_{k}}(z)\ $where $\Gamma $ is the
incomplete gamma function. The total number of voters is $%
N_{T}=\sum_{k=1}^{K}N_{k}$, which, by the aggregation property of the
Poisson distribution \citep{johnson1992univariate}, is also a Poisson random
variable with mean $n_{T}=\sum_{k=1}^{K}n_{k}$. To avoid confusion, we
denote the expected size of the total population with a subscript $T$.

Let $p_{k}$ represent the average probability that members of group $k$ vote
for $\mathcal{A}$, and let $q_{k}$ represent the probability that members of
$k$ vote for party $\mathcal{B}$. By the decomposition property of the
Poisson distribution, $A_{k}$, the number of votes
for party $\mathcal{A}$ in group $k$, is a Poisson random variable with mean
$\lambda _{k}=p_{k}n_{k}$. Let $\gamma_k$  represent the corresponding expected votes for $\mathcal{B}$. We use the notation $(p,q)=((p_{1},q_{1}),\dots
,(p_{K},q_{K}))$ as the profile of vote probabilities. We denote the
profile of expected votes for parties $\mathcal{A}$ and $\mathcal{B}$ as $(\lambda ,\gamma
)=((\lambda _{1},\gamma _{1}),\dots ,(\lambda _{K},\gamma _{K}))$ and $%
(A,B)=((A_{1},B_{1}),\dots ,(A_{K},B_{K}))$ as the profile of actual votes.
Party $\mathcal{A}$ wins the election if it receives more votes than party $%
\mathcal{B}$ ($\sum_{k=1}^{K}A_{k}>\sum_{k=1}^{K}B_{k}$); ties are resolved
by a coin flip. We characterize profiles of vote
probabilities that can be supported in Nash equilibrium, show how these
equilibria vary within a winner-take all environment and examine the
incentives this creates for group formation and maintenance.

Voters care both about policy benefits and any potential prizes or punishments the parties
distribute. With regard to policy benefits, voter $i$ receives a policy
reward of $\varepsilon_{i}$ if party $\mathcal{A}$ wins the election and a
policy payoff of $0$ if $\mathcal{B}$ wins. The random variable $\varepsilon
_{i}$ represents individual $i$'s private evaluation of party $\mathcal{A}$
relative to party $\mathcal{B}$. We assume the individual evaluations are
independently identically distributed with distribution $\Pr (\varepsilon
_{i}<r)=G(r)$. As a preview, equilibria are characterized by thresholds, $\tau_{Ak}$ for instance, such that voter $i$ in group $k$ votes for 
$\mathcal{A}$ if  $\varepsilon_{i}>\tau_{Ak}$ and $p_{k}=1-G(\tau_{Ak})$.

In addition to personal policy gains, individuals care about the benefits or punishments
that parties might provide to their group. The concern
here is with allocation mechanisms rather than on what is being
allocated. Hence, rather than work with the litany of titles for benefits we
simply refer to all preferential rewards as \emph{prizes}, $\zeta $ and all group specific punishments as \emph{penalties}, $\chi$. When
necessary, we label the value of these prizes and penalties as $\zeta _{A}$, $\zeta _{B}$, $\chi_{A}$ and $\chi_{B}$
according to which party hands them out. What is essential for our model is
that parties can observe the level of political support from each group and
that there exists a means of preferentially rewarding or punishing groups.

We focus primarily on prizes and explore the non-rival versus rival nature of prizes. As it happens, this
factor influences the optimal division of society into groups from the
perspective of political parties and citizens. Although in practice all
policies have private and public goods components, we contrast the limiting
cases. We treat a prize as a non-rival local public good (or a pure club good)
if its cost of provision is unrelated to the size of the group that benefits
from it. We refer to this first case, where the marginal cost of increasing group size is zero, as a
non-rival prize. Prizes based on private goods are rival and they have a
constant marginal cost of providing the prize as group size increases.
However, until we examine the relative cost of prize provisions under
different arrangements of groups, the essential point is that the members of
the group to which the prize is allocated get benefits worth $\zeta $. With
this setup in mind, we explore how parties can condition their distribution
of prizes on the vote outcome $(A,B)$.

We refer to the mechanism parties use to distribute rewards as Contingent
Prize Allocation Rules (CPAR). Let $GA_{k}(A,B)$ be the probability that party $\mathcal{A}$ awards the prize to group $k$ if the vote profile
is $(A,B)$. Although we develop the logic of our arguments with respect to
party $\mathcal{A}$, throughout there are parallel considerations with
respect to party $\mathcal{B}$. That is, each party allocates prizes. Although there are many plausible CPARs, we focus here on the common Winner-Takes-All (WTA) Rule. \citet{bdmsmith2015effort} examine a broader class of CPARs in a related setting. 
Under this rule party $\mathcal{A}$ rewards the
most supportive group (or groups). Other groups receive nothing.
\begin{equation*}
GA_{k}(A,B)=\left\{
\begin{array}{ccc}
1 & if & A_{k}=\max \{A_{1},\dots ,A_{K}\}>0 \\
0 &  & otherwise%
\end{array}%
\right.
\end{equation*}
Note that in event of a tie for most supportive group, we assume both (or more) groups receive a prize.  We assume that party $\mathcal{A}$ punishes group $k$ if and only if it is the unique least supportive group:  
\begin{equation*}
HA_{k}(A,B)=\left\{
\begin{array}{ccc}
1 & if & A_{k}<\min \{A_{1},\dots , A_{k-1},A_{k+1},\dots,A_{K}\} \\ 
0 &  & otherwise%
\end{array}%
\right.
\end{equation*}

\section{Pivotality and Voting}

 The standard concept of voter pivotality is the likelihood of shifting
the outcome of an election from one party to another. We refer to this as the outcome
pivot, $OP_{A}$, which is defined formally below. Voters can also be
 pivotal in terms of the distribution of the prize or punishment. That
is, by voting for party $\mathcal{A}$, a voter may not only increase the
likelihood that party $\mathcal{A}$ wins; she also increases the probability
that her group will be the most supportive group and reduces the probability that her group is the least supportive group.  We refer to the likelihood of being pivotal in terms of prize allocation as the Prize Pivot, $PP_{A}$ and the likelihood of being pivotal in terms of punishment allocation as the Penalty Pivot, $QP_{A}$. In all
cases we define analogous terms with respect to party $\mathcal{B}$.

Following from the \textit{environmental equivalence} result of Myerson (1998, Theorem 2), from the perspective of each member of group $k$, the other $N_{k}-1$ members of $k$ can also be assumed to be Poisson distributed with mean $n_{k}$. This feature makes the Poisson framework especially
attractive for modeling pivotality  as the voter's and analyst's assessment of other voters coincide. 

The proposition below provides a definition and calculation of Outcome Pivot, $OP_{A}$.
 Given vote probability profile $%
(p,q)$, the number of votes for party $\mathcal{A}$ in district $k$ is a
Poisson random variable with mean $\lambda_{k}=p_{k}n_{k}$ and the total
number of votes for $\mathcal{A}$ is also a Poisson random variable $A\ $%
with mean $n_{T}p=\sum_{j=1}^{K}p_{j}n_{j}$, where $n_{T}=\sum%
\limits_{j=1}^{K}n_{j}$ and $p$ is the weighted average probability of
voting for $\mathcal{A}$. Analogously, the total number of votes for party $%
\mathcal{B}$ is $B$, a Poisson random variable with mean $%
n_{T}q=\sum_{j=1}^{K}q_{j}n_{j}$. Given the well-known result that an
individual's vote only influences who wins if it breaks a tie or turns a
loss into a draw, the
proposition below defines and characterizes $OP_{A}$.\footnote{$OP_B$ is analogously defined as $
OP_{B} =\Pr (\mbox{$\mathcal B$ wins $|$ voter $i$ abstains})-\Pr (\mbox{$\mathcal B$ wins $|$ voter $i$ votes $\mathcal B$})< 0$.}

\begin{proposition}
\label{OPA_define} Given the vote probability profile $(p,q)$,
\begin{eqnarray}
OP_{A} &=&\Pr (\mbox{$\mathcal A$ wins $|$ voter $i$ votes $\mathcal A$}%
)-\Pr (\mbox{$\mathcal A$ wins $|$ voter $i$ abstains})  \label{OPA} \\
&=&\frac{1}{2}\Pr (A=B)+\frac{1}{2}\Pr (A=B-1)  \notag \\
&=&e^{-n_{T}(p+q)}\frac{1}{2}(I_{0}(2n_{T}\sqrt{pq})+(\frac{q}{p})^{\frac{1}{%
2}}I_{1}(2n_{T}\sqrt{pq}))\notag 
\end{eqnarray}

where $A=\sum_{k=1}^{K}A_{k}$, \ $B=\sum_{k=1}^{K}B_{k}$ and $I_{m}(x)$ is
the modified Bessel function of the first kind.
\end{proposition}

\begin{proof}
From \citet{skellam1946frequency}, if $A$ and $B$ are Poisson random variables with means $%
n_{T}p$ and $n_{T}q$ respectively, then $Sk(n_{T}p,n_{T}q,m)=\Pr
(A-B=m)=\allowbreak e^{-(n_{T}p+n_{T}q)}(\frac{n_{T}p}{n_{T}q})^{\frac{m}{2}%
}I_{|m|}(2n_{T}\sqrt{pq})$, where $I_{m}$ is the modified Bessel function of
the first kind. The function $Sk$ is the Skellam distribution with
parameters $n_{T}p$ and $n_{T}q$. Therefore $OP_{A}$ is simply the average
of the Skellam distribution evaluated at $m=0$ and $m=-1$. So $%
OP_{A}=\allowbreak e^{-n_{T}(p+q)}((\frac{p}{q})^{\frac{0}{2}}\frac{1}{2}%
(I_{0}(2n_{T}\sqrt{pq})+(\frac{q}{p})^{\frac{1}{2}}I_{1}(2n_{T}\sqrt{pq}))$.
The Outcome Pivot with respect to voting for party $\mathcal{B}$ is
analogously defined, $OP_{B}$.
\end{proof}

Voters not only affect which party wins but also the distribution of prizes.
Prize Pivot, $PP_{A,k}(p,q)$, refers to the change in the probability that party $\mathcal{A}$ allocates the prize to group $k$ if a member of $k$ votes
for $\mathcal{A}$ rather than abstains and the vote probability profile is $%
(p,q)$. To simplify notation we generally omit the profile $(p,q)$.

\begin{proposition}
\label{PPA_define} The prize pivot, $PP_{A,k}$, and the penalty pivot, $QP_{A,k}$, for any individual voter in
group $k$ are:
\begin{eqnarray}
PP_{A,k} &=&Pr(\mbox{Prize}|\mbox{vote}\mathcal{A})-Pr(\mbox{Prize}|%
\mbox{abstain})  \label{PPAdefine} \\
&=& \sum_{a=0}^{\infty }f_{n_{k}p_{k}}(a)\left( \prod\limits_{j\neq
k}F_{n_{j}p_{j}}(a+1)-\prod\limits_{j\neq k}F_{n_{j}p_{j}}(a)\right) \geq 0 \notag
\end{eqnarray}
\begin{eqnarray}
QP_{A,k} &=&Pr(\mbox{Punishment}|\mbox{vote}\mathcal{A})-Pr(\mbox{Punishment}|%
\mbox{abstain})  \label{QPAdefine} \\
&=& \sum_{a=0}^{\infty }f_{p _{k}n_k}(a)\left( \
\prod\limits_{j\neq k}(1-F_{p _{j}n_j}(a+1))-\prod\limits_{j\neq
k}(1-F_{p_{j}n_j}(a))\right) \leq 0 \notag
\end{eqnarray}
\end{proposition}

\begin{proof}
Suppose $A_{k}=a_{k}$. If a voter in group $k$ abstains then her group
receives the prize $\zeta $ if $a_{k}\geq \max \{A_{j\neq k}\}$. Since $%
A_{j} $ is Poisson distributed with mean $n_{j}p_{j}$, $\Pr (a_{k}\leq
A_{j})=F_{n_{j}p_{j}}(a_{k})$ and the probability that $a_{k}$ is the
maximum of all groups' support for $\mathcal{A}$ is $\prod\limits_{j\neq
k}\Pr (A_{j}\leq a_{k})=\prod\limits_{j\neq k}F_{n_{j}p_{j}}(a_{k})$. Since $%
A_{k}$ is Poisson distributed, group $k$'s probability of receiving the prize if the voter
abstains is $\sum_{a=0}^{\infty }f_{n_{k}p_{k}}(a)\prod\limits_{j\neq
k}F_{n_{j}p_{j}}(a)$.

If the voter votes for $\mathcal{A}$, then $\Pr (a_{k}+1\geq \max \{A_{j\neq
k}\})=\prod\limits_{j\neq k}\Pr (A_{j}\leq 1+a_{k})=\prod\limits_{j\neq k}F_{n_{j}p_{j}}(a_{k}+1)$ and the probability $k$ receives the prize is 
$\sum_{a=0}^{\infty }f_{n_{k}p_{k}}(a)\prod\limits_{j\neq
k}F_{n_{j}p_{j}}(a+1)$. Therefore $PP_{A,k}= \sum_{a=0}^{\infty
}f_{n_{k}p_{k}}(a)(\prod\limits_{j\neq
k}F_{n_{j}p_{j}}(a+1)-\prod\limits_{j\neq k}F_{n_{j}p_{j}}(a))$. Further, $%
PP_{A,k}$ is continuous in all components of $(p,q)$ because the underlying
Poisson distributions are continuous.

The derivation of the punishment pivot is analogous. If $A_{k}=a$ and the voter votes for $\mathcal{A}$, then group $k$ is punished if and only if $a+1<\min (A_{j\neq k})$, which happens with probability $\prod\limits_{j\neq k}(1-F_{p_{j}n_{j}}(a+1))$. If the voter abstains, then group $k$ is punished with probability $\prod\limits_{j\neq k}(1-F_{p_{j}n_{j}}(a))$. Summing over all possible $A_k$'s produces equation \ref{QPAdefine}.
\end{proof}

The model assumes parties distribute prizes/punishment whether they win or lose the election. This might be a reasonable assumption in a federal system or if the prizes are access to party level resources. In other settings, parties in office might deliver larger prizes than those parties excluded from access to government resources. Prize pivots are more complicated in such settings.\footnote{Although prize pivots vary if only victorious parties issue prizes, in many important cases the distinctions are easily handled. First, in non-competitive equilibria, those in which the expected vote share for one party is much higher than the other, one party is virtually certain to win and the two concepts of prize pivot converge. Second, in symmetric cases of competitive elections, each party is equally likely to win and so the prize pivot would be half the value calculated here.}  

Next we characterize how pivot probabilities differ across groups as a function of the expected turnouts $\lambda_{1}, \dots , \lambda_{K}$. We say group $i$ is $\mathcal{A}$-active if $\lambda_i>0$ and let $W_{A}=\{i\in \{1,..,K\}:p_{i}>0\}$ represent the set of such groups.
\begin{proposition}
\label{useful_lemma} \begin{eqnarray}
\Delta P&=&PP_{A,j}-PP_{A,k}  \label{DeltaP} \\
&=& \sum_{a=0}^{\infty }([f_{\lambda _{j}}(a)f_{\lambda
_{k}}(a+1)-f_{\lambda _{_{k}}}(a)f_{\lambda _{j}}(a+1)]\prod\limits_{i\neq
j,k}F_{\lambda _{i}}(a))  \notag \\
&&+ \sum_{a=0}^{\infty }((f_{\lambda _{j}}(a)f_{\lambda
_{k}}(a+1)-f_{\lambda _{k}}(a)f_{\lambda _{j}}(a+1))\prod\limits_{i\neq
j,k}f_{\lambda _{i}}(a+1))  \notag \\
&&+ \sum_{a=0}^{\infty }([f_{\lambda _{j}}(a)F_{\lambda
_{k}}(a)-f_{\lambda _{k}}(a)F_{\lambda _{j}}(a)]\prod\limits_{i\neq
j,k}f_{\lambda _{i}}(a+1))  \notag
\end{eqnarray}
Further, if there are only two groups with positive turnout, then the group with the smaller expected turnout has the larger pivot probability (if $\lambda_i=0$ for all $i\neq j,k$ and $\lambda _{j}>\lambda
_{k} $, then $PP_{A,j}<PP_{A,k}$).
\end{proposition}

\begin{proof}
Let $M(a)$ represent the distribution of the greatest number of votes for $%
\mathcal{A}$ by any groups other than $j$ or $k$: $M(a)=\Pr (Max_{i\neq
j,k}\{A_{i}\}\leq a)=\prod\limits_{i\neq j,k}F_{\lambda_{i}}(a)$. Let $m(a)$
be the associated probability mass function. Note that, if $n_{i}p_{i}=0$,
then $M(a)=1$ and $m(a)=0$ for all $a>0$.

Noting that $M(a+1)=M(a)+m(a+1)$ and $F(a+1)=F(a)+f(a+1)$, the prize pivot
for group $j$ can be written as,

\begin{eqnarray}
PP_{A,j} &=& \sum_{a=0}^{\infty }f_{\lambda _{j}}(a)[M(a+1)F_{\lambda
_{k}}(a+1)-M(a)F_{\lambda _{k}}(a)]  \label{pp_general} \\
&=& \sum_{a=0}^{\infty }f_{\lambda _{j}}(a)[(M(a)+m(a+1))(F_{\lambda
_{k}}(a)+f_{\lambda _{k}}(a+1))-M(a)F_{\lambda _{k}}(a)]  \notag \\
&=& \sum_{a=0}^{\infty }f_{\lambda _{j}}(a)\left( M(a)f_{\lambda
_{k}}(a+1)+m(a+1)F_{\lambda _{k}}(a)+f_{\lambda _{K}}(a+1)m(a+1)\right)
\notag
\end{eqnarray}

The difference between $PP_{A,j}$ and $PP_{A,k}$ is:
\begin{eqnarray}
\Delta P &=&PP_{A,j}-PP_{A,k}  \label{DELTA} \\
&=& \sum_{a=0}^{\infty }M(a)[f_{\lambda _{j}}(a)f_{\lambda
_{k}}(a+1)-f_{\lambda _{_{k}}}(a)f_{\lambda _{j}}(a+1)]  \notag \\
&&+ \sum_{a=0}^{\infty }m(a+1)(f_{\lambda _{j}}(a)f_{\lambda
_{k}}(a+1)-f_{\lambda _{k}}(a)f_{\lambda _{j}}(a+1))  \notag \\
&&+ \sum_{a=0}^{\infty }m(a+1)[f_{\lambda _{j}}(a)F_{\lambda
_{k}}(a)-f_{\lambda _{k}}(a)F_{\lambda _{j}}(a)]  \notag
\end{eqnarray}

The first term and second terms contain
\begin{equation*}
\lbrack f_{\lambda_{j}}(a)f_{\lambda_{k}}(a+1)-f_{\lambda
_{_{k}}}(a)f_{\lambda_{j}}(a+1)]=\frac{\lambda_{j}^{a}\lambda
_{k}^{a}e^{-\lambda_{k}}e^{-\lambda_{j}}}{a!a!}(\lambda_{k}-\lambda _{j})<0
\end{equation*}

Hence if $M(0)=1$ then $\Delta <0$.

The third term of $\Delta $ contains the expression $[f_{\lambda
_{j}}(a)(F_{\lambda_{k}}(a))-f_{\lambda_{k}}(a)(F_{\lambda_{j}}(a))]$ which
can be written as

\begin{eqnarray}
&&\frac{\lambda_{j}^{a}e^{-\lambda_{j}}}{a!}\sum_{x=0}^{a}\frac{\lambda
_{k}^{x}e^{-\lambda_{k}}}{x!}-\frac{\lambda_{k}^{a}e^{-\lambda_{k}}}{a!}%
\sum_{x=0}^{a}\frac{\lambda_{j}^{x}e^{-\lambda_{j}}}{x!}  \notag \\
&=&\frac{e^{-\lambda_{j}}e^{-\lambda_{k}}}{a!}\sum_{x=0}^{a}\frac{\lambda
_{j}^{x}\lambda_{k}^{x}}{x!}(\lambda_{j}^{a-x}-\lambda_{k}^{a-x})
\end{eqnarray}

Since $(\lambda_{j}^{a-x}-\lambda_{k}^{a-x})>0$, the third term of $\Delta $
is positive, so $\Delta $ cannot be definitively signed if $M(0)<1$.
Substitution of $M(a)$ and $m(a)$ into equation \ref{DELTA} produces
expression \ref{DeltaP}.
\end{proof}

The analogous result in terms of penalty pivots is: 	
\begin{proposition} \label{prop:deltaQP} If all groups have positive turnout, then the ordering of magnitudes of the penalty pivots of the groups are opposite to the ordering of the expected turnouts of the groups: $\lambda _{i}>0$ for all $i\in K$ and $\lambda _{j}>\lambda _{k}$, then $|QP_{A,k}|>|QP_{A,j}|$.
	\end{proposition} 
\begin{proof}	
If a representative voter in $j$ abstains, then group $j$ is punished if $A_{j}<\min \{A_{i}$ for all $i\neq k\}$. If the voter votes for $\mathcal{A}$, then group $j$ is punished if $A_{j}+1<\min \{A_{i}$ for all $i\neq k\}$. 
Let $R(a)=\Pr (\min \{A_{i}|i\neq j,k\}>a)=\prod\limits_{i\neq k,j}(1-F_{\lambda _{i}}(a))$ be the probability that all groups outside of $j$ and 
$k$ produce more than $a$ votes for party $\mathcal{A}$ and $R(a+1)=\Pr (\min \{A_{i}|i\neq j,k\}>a+1)=\prod\limits_{i\neq k,j}(1-F_{\lambda _{i}}(a+1))$. 
Further let $r(a+1)=R(a+1)-R(a) < 0$.
	
Utilizing that $f_{\lambda _{k}}(a+1)=f_{\lambda _{k}}(a)\frac{\lambda _{k}}{a+1}$ we write 
\begin{eqnarray*}
QP_{A,j}&=&\sum_{a=0}^{\infty }f_{\lambda _{j}}(a)\left[ (1-F_{\lambda _{k}}(a+1))R(a+1)-(1-F_{\lambda _{k}}(a))R(a)\right] \\
&=& \sum_{a=0}^{\infty }f_{\lambda _{j}}(a)\left[ r(a+1)(1-F_{\lambda _{k}}(a)-f_{k}(a)\frac{\lambda _{k}}{a+1})-f_{k}(a)\frac{\lambda _{k}}{a+1}R(a)\right] 
\end{eqnarray*}
Therefore 
\begin{eqnarray*}
\Delta Q&=&QP_{A,j}-QP_{A,k} \\
&=&\sum_{a=0}^{\infty }R(a)\frac{f_{\lambda _{j}}(a)f_{\lambda _{k}}(a)}{a+1}(\lambda _{j}-\lambda _{k})+\sum_{a=0}^{\infty }r(a+1)\left[ f_{\lambda _{j}}(a)(1-F_{\lambda _{k}}(a+1))-f_{\lambda _{k}}(a)(1-F_{\lambda _{j}}(a+1))\right] 
\end{eqnarray*}	
Since $\lambda _{j} > \lambda _{k}$, the first summation is positive. Using the substitution that $\lambda _{j}=\rho \lambda _{k}$ and $(1-F_{\lambda _{k}}(a+1))=\sum_{x=a+1}^{\infty }\frac{e^{-\lambda _{k}}\lambda _{k}^{x}}{x!}$, the term  $f_{\lambda _{j}}(a)(1-F_{\lambda _{k}}(a+1))-f_{\lambda _{k}}(a)(1-F_{\lambda _{j}}(a+1))$ can be written as $e^{-\lambda _{j}-\lambda _{k}}\sum_{x=a+1}^{\infty }\frac{\lambda _{k}^{a+x}}{a!x!}(\rho ^{a}-\rho ^{x})$. Since $\rho>1$ and $r(a+1) < 0 $, the second summation is also positive. Therefore, $\Delta Q<0$.
\end{proof}

In standard pivotal voting games without prizes there is an underdog effect \citep{kartal2014comparative}.  As \citet{taylor2010unified} show, when there is costly voting the minority group turns out at a higher rate; but it still tends to lose to the majority. It is worth noting that \citet{campbell1999}  and \citet{krishna2011overcoming} show that if a minority has greater salience for an issue they can overcome the majority. However, absent such a systematic bias in salience, the underdog effect predicts that small groups try harder, although they are still likely to lose. We observe a similar pattern with respect to incentives created by prize and penalty pivots. 

Provided that there are not more that 2 $\mathcal{A}$-active groups, proposition \ref{useful_lemma} states that the group with the lower turnout has the higher prize pivot --the underdog effect. Proposition \ref{prop:deltaQP} exhibits a similar underdog effect with respect to penalty pivots; the magnitude of the penalty pivot is largest for the group with the smallest expected turnout.\footnote{Note that if $\lambda_i=0$ for some group, then $QP_{A,j}=0$ for all $j \neq i$.} There is a negative feedback with respect to prize and penalty motivations. As a group increases its expected turnout, that group's influence over the distribution of prizes and penalties diminishes and such negative feedback reduces the incentive to further increase turnout and creates stability. The stability induced by the underdog persists for any number of groups (so long as they are all $\mathcal{A}$-active) for penalty pivots. However, prize pivots are unstable with more than 2 $\mathcal{A}$-active groups. 

Proposition \ref{useful_lemma} shows that if $\lambda _{j}>\lambda
_{k} $, then $PP_{A,j}<PP_{A,k}$. However the result only holds if $\lambda_i=0$ for all other groups. As the following example illustrates, when there are 3 or more $\mathcal{A}$-active groups, stability induced by the underdog effect breaks down and there is positive feedback from increasing group turnout. It is worth noting that the tournament between a large number of employees by \citet{rosen1986prizes} is constructed as a series of binary contests.

Figure~\ref{FigAsym.pdf} examines the case of three $\mathcal{A}$-active groups and plots the prize pivot for each group as asymmetry in group turnout is introduced. In the figure group 2 has
constant turnout of 1000 voters, $\lambda _{2}=1000$ and the figure plots
the prize pivot for each of the three groups as turnout in group 1 increases
above $1000$ and turnout in group 3 decreases below 1000. In particular, the
figure is constructed assuming that $\lambda _{1}=(1+\delta )\lambda _{2}$
and $\lambda _{3}=(1-\delta )\lambda _{2}$. The parameter $\delta $
represents the degree of asymmetry between expected group turnouts. On the
left hand side of the figure ($\delta =0$); all the groups have the same
expected turnout and hence the same prize pivot. On the right hand side of
the figure, group 1's turnout is 1\% larger than group 2's, which in turn is
about 1\% larger than group 3's ($\delta =.01$). As asymmetry in expected
turnout increases, prize pivots diverge, and the group with the smallest expected turnout has the smallest prize pivot. To illustrate the logic behind this result consider the conditions that make the smallest group, 3 in the example, pivotal. In rough terms, to be pivotal group 3 members need to match the votes of group 2 and have more votes than group 1, or match group 1 and beat group 2. Proposition \ref{useful_lemma}  tells us that with respect to matching the votes of group 2, members of group 3 are more pivotal than members of group 2. However, in the three group case, the probability of generating more votes than the larger group 1 must also be taken into account; and since group 3 is the smallest, this factor reduces the prize pivotality of group 3. 

Figure~\ref{FigAsym.pdf} about here

Although figure~\ref{FigAsym.pdf} is only an illustrative example, it reflects the positive feedback induced in prize pivots when there are three or more $\mathcal{A}$-active groups. We examine this instability more systematically in the appendix.  As the number of voters grows large, there are simple approximations for the pivots, as we next examine. 

\subsection{Asymptotic Approximations of Pivots}

Outcome and prize pivots both arise and are of interest in settings with
very large numbers of prospective voters. Therefore we use asymptotic
approximations of the Bessel function to generate reliable estimates of
pivot probabilities. To derive these approximations we assume the expected
number of voters, $n_{k}p_{k}$, is relatively large.

As the expected number of voters converges toward infinity, the
approximations converge to their true values. We indicate the accuracy of
these approximations in finite populations.

\begin{proposition}
\label{pro_OP} Outcome Pivot: $OP_{A}$: If $n_{T}=\sum_{k=1}^{K}n_{k}$, $p=%
\frac{1}{n_{T}}\sum_{k=1}^{K}p_{k}n_{k}$ and $q=\frac{1}{n_{T}}%
\sum_{k=1}^{K}q_{k}n_{k}$ then
\begin{equation}
OP_{A}\sim \widetilde{OP}_{A}=\frac{1}{2\sqrt[4]{pq}\sqrt{\pi n_{T}}}\frac{%
\sqrt{p}+\sqrt{q}}{2\sqrt{p}}\cdot e^{-n_{T}\left( \sqrt{p}-\sqrt{q}\right)
^{2}}  \label{OPAapprox}
\end{equation}
\end{proposition}

This is the same approximation used by Myerson (1998), so we provide only a
brief sketch. As derived above, the difference between the vote for $%
\mathcal{A}$ and $\mathcal{B}$ is Skellam distributed: $%
OP_{A}=e^{-n_{T}(p+q)}((\frac{p}{q})^{\frac{m}{2}}\frac{1}{2}(I_{0}(2n\sqrt{%
pq})+(\frac{q}{p})^{\frac{1}{2}}I_{1}(2n_{T}\sqrt{pq}))$. The modified
Bessel function of the first kind, $I_{m}(x)$ is a well known mathematical
function that for fixed $m$ and large $x$ is well approximated (\citealt{abramowitz1965handbook} p. 377:)
\begin{eqnarray*}
I_{|m|}(x) &\sim &\frac{e^{x}}{\sqrt{2\pi x}}(1-\frac{4m^{2}-1}{8x}+\frac{%
(4m^{2}-1)(4m^{2}-9)}{2!(8x)^{2}} \\
&&-\frac{(4m^{2}-1)(4m^{2}-9)(4m^{2}-25)}{3!(8x)^{3}}+\dots )
\end{eqnarray*}

We use the first term of this approximation $I_{|m|}(x)\sim \frac{e^{x}}{%
\sqrt{2\pi x}}$ and equation \ref{OPAapprox} follows directly. To check the
accuracy we evaluate $(I_{m}(x)-\frac{e^{x}}{\sqrt{2\pi x}})/I_{m}(x)$ for $%
|m|=0,1$. Ninety-nine percent accuracy is attained when $x>38.2$. The
approximation improves as $x$ increases. For large populations the outcome pivot estimates are accurate. For
instance, if the population mean is $n_{T}=100,000$ and voters support
parties $\mathcal{A}$ and $\mathcal{B}$ with probability $p=.5$ and $q=.5$,
then the approximation error for the outcome pivot is around .0001\%.

Asymptotic approximations of prize pivots exist. The proposition below characterizes an approximation
of $PP_{A,1}$ when there are two $\mathcal{A}$-active groups. We provide an online appendix that
characterizes a series of approximations of prize pivots for three or more $%
\mathcal{A}$ -active groups, the simplest of which relies upon the expansion
of the products in equation \ref{PPAdefine}, a Gaussian approximation
of the Poisson distribution and Laplace's method of integration.

\begin{proposition}
\label{pro_PP} Prize Pivot and Penalty Pivot Approximation:

If there are $2$ $\mathcal{A}$-active groups with expected votes for $%
\mathcal{A}$ of $\lambda _{1}$ and $\lambda _{2}$, then as $\lambda
_{i}\rightarrow \infty $, 
\begin{equation*} PP_{A,1}\sim \frac{1}{2\sqrt{\pi }\sqrt[4]{%
\lambda _{1}\lambda _{2}}}\frac{\sqrt{\lambda _{2}}}{\sqrt{\lambda _{1}}}%
e^{-(\sqrt{\lambda _{1}}-\sqrt{\lambda _{2}})^{2}}
\end{equation*}

If there are only $2$ groups with expected votes for $%
\mathcal{A}$ of $\lambda _{1}$ and $\lambda _{2}$, then as $\lambda
_{i}\rightarrow \infty $, 
\begin{equation*} QP_{A,1}\sim -\frac{1}{2\sqrt{\pi }\sqrt[4]{%
\lambda _{1}\lambda _{2}}}\frac{\sqrt{\lambda _{2}}}{\sqrt{\lambda _{1}}}%
e^{-(\sqrt{\lambda _{1}}-\sqrt{\lambda _{2}})^{2}}
\end{equation*}

\end{proposition}

\begin{proof}
$PP_{A,1}= \Pr (A_{1}=A_{2}+1)$. $A_{1}$ and $A_{2}$ are Poisson random
variables with means $\lambda _{1}$ and $\lambda _{2}$, the difference $%
A_{1}-A_{2}$ is Skellam distributed: $\Pr (A_{1}-A_{2}=1)=Sk(\lambda
_{1},\lambda _{2},1)=\allowbreak e^{-(\lambda _{1}+\lambda _{2})}\sqrt{%
\frac{\lambda _{2}}{\lambda _{1}}}I_{|1|}(2\sqrt{\lambda _{1}\lambda _{2}})$%
, where $I_{m}$ is the modified Bessel function of the first kind. Using the
approximation (discussed above) that $I(x)\sim \frac{e^{x}}{\sqrt{2\pi x}}$,
$PP_{A,1}\sim e^{-(\lambda _{1}+\lambda _{2})}\sqrt{\frac{\lambda _{2}}{%
\lambda _{1}}}\frac{e^{2\sqrt{\lambda _{1}\lambda _{2}}}}{\sqrt{2\pi 2\sqrt{%
\lambda _{1}\lambda _{2}}}}=\frac{1}{2\sqrt{\pi }\sqrt[4]{\lambda
_{1}\lambda _{2}}}\frac{\sqrt{\lambda _{2}}}{\sqrt{\lambda _{1}}}e^{-(\sqrt{%
\lambda _{1}}-\sqrt{\lambda _{2}})^{2}}$.

Analogously, $QP_{A,1}=-\Pr(A_1=A_2-1)$. Since $\Pr(A_1-A_2=-1)=Sk(\lambda
_{1},\lambda _{2},-1)$ and the same approximation holds.
\end{proof}

\subsection{Voting Calculus}

Suppose we consider any fixed vote profile $%
(p,q)=((p_{1},q_{1}),(p_{2},q_{2}),\dots ,(p_{K},q_{K}))$ that describes the
probability with which members of each group support $\mathcal{A}$ and $%
\mathcal{B}$ respectively. Given this profile, the following equations
characterize individual private evaluations of party $\mathcal{A}$ relative
to $\mathcal{B}$ (that is, the value of $\varepsilon _{i}$) such that an
individual in group $k$ is indifferent between her vote choices.
\begin{equation}
U_{k}(vote\mathcal{A})-U_{k}(abstain)=(\tau _{Ak})OP_{A}+\zeta_{A}PP_{A,k}-\chi_{A}QP_{A,k}-c=0
\label{AvsAbstain}
\end{equation}%
\begin{equation}
U_{k}(vote\mathcal{B})-U_{k}(abstain)=(\tau _{Bk})OP_{B}+\zeta_{B}PP_{B,k}-\chi_{B}QP_{B,k}-c=0
\label{BvsAbstain}
\end{equation}%
\begin{equation}
\begin{array}{c}
U_{k}(voteA)-U_{k}(voteB)=(\tau _{ABk})(OP_{A}-OP_{B})+ \\
\zeta_{A} PP_{A,k}-\zeta_{B}PP_{B,k}-\chi_{A}QP_{A,k}+\chi_{B}QP_{B,k}=0%
\end{array}
\label{AversusB}
\end{equation}

The thresholds, $\tau_{Ak}$, $\tau_{Bk}$ and $\tau_{ABk}$ that solve these
equations characterize Nash equilibria.

\begin{theorem}
\label{existence}There exist vote probability profiles $(p,q)$ supported by
Nash equilibrium voting behavior: voter $i$ in group $k$ votes for party $%
\mathcal{A}$ if $\varepsilon_{i}>\max \{\tau_{Ak},\tau_{ABk}\}$; votes for $%
\mathcal{B}$ if $\varepsilon_{i}<\min \{\tau_{Bk},\tau_{ABk}\}$ and abstains
if $\min \{\tau_{Bk},\tau_{ABk}\}<\varepsilon_{i}<\max \{\tau
_{Ak},\tau_{ABk}\}$. The thresholds, $\tau_{Ak},$ $\tau_{Bk}$ and $\tau
_{ABk}$, solve equations \ref{AvsAbstain}, \ref{BvsAbstain}, and \ref%
{AversusB} for each group and $p_{k}=1-G(\max \{\tau_{Ak},\tau_{ABk}\})$ and
$q_{k}=G(\min \{\tau_{Bk},\tau_{ABk}\})$.
\end{theorem}

\begin{proof}
Given the Poisson population assumption, there is always some, albeit very
small, probability that $i$ is the only voter. In such a setting, her vote
would determine the outcome. This ensures that $OP_{A}>0$ and $OP_{B}<0$.
Therefore equation \ref{AvsAbstain} is an increasing linear function of $%
\tau_{Ak}$. Therefore for any given vote profile $(p,q)$, there is a unique
threshold that solves the equation (and the same for equations \ref%
{BvsAbstain} and \ref{AversusB}). These three equations correspond to
differences in expected value from each of the voter's actions. If $%
\varepsilon_{i}>\max \{\tau_{Ak},\tau_{ABk}\}$ then $i$ votes for $\mathcal{A%
}$ since $U_{k}(vote\mathcal{A})>U_{k}(abstain)$ and $U_{k}(vote\mathcal{A}%
)>U_{k}(vote\mathcal{B})$. Similarly if $\varepsilon_{i}<\min
\{\tau_{Bk},\tau_{ABk}\}$, then $i$ votes for $\mathcal{B}$.\newline
Given the thresholds, an individual in group $k$ votes for $\mathcal{A}$
with probability $\widetilde{p}_{k}(p,q)=1-G(\max \{\tau_{Ak},\tau _{ABk}\})
$ and votes for $\mathcal{B}$ with probability $\widetilde{q}%
_{k}(p,q)=G(\min \{\tau_{Bk},\tau_{ABk}\})$. Since outcome, prize and penalty
pivots are continuous in all components of the vote profile $(p,q)$, the $%
\tau $ thresholds, and hence $\widetilde{p}_{k}(p,q)$ and $\widetilde{q}%
_{k}(p,q)$, are continuous in all components of the vote profile. Let $%
M:[0,1]^{2K}\rightarrow \lbrack 0,1]^{2K}$ be this best response function
for all the groups. That is to say, $M$ maps $(p,q)$ into simultaneous best
responses for all groups $(\widetilde{p},\widetilde{q})=((\widetilde{p}%
_{1}(p,q),\widetilde{q}_{1}(p,q)),\dots,(\widetilde{p}_{K}(p,q),\widetilde{q}%
_{K}(p,q)))$. As $M$ is continuous and maps a compact set back into itself,
by Brouwer's fixed point theorem, a fixed point exists.\footnote{If the function $G$
is discontinuous then $\widetilde{p}_{k}(p,q)$ is an upper hemi-continuous
mapping, voters randomize at the indifference points and a fixed point
exists by Kakutani's fixed point theorem.}
\end{proof}

\section{Competition for Prizes}

The literature focuses on the case where there are neither prizes nor punishments and voters are
pivotal only in terms of which party wins (Ledyard (1984); 
Krishna and Morgan (2012); Myerson (1998); \cite{feddersen1997voting}; \cite{acharya2015equilibrium}). In that policy-only case, turnout is relatively
low and elections are close as expected vote shares are similar. We contrast
that case with the other limiting case in which there are no policy
differences between the parties and groups of voters compete solely for the
prize or to avoid punishment. In this setting the electoral outcome is lopsided. However, the competition between the groups for prizes is close. 

\begin{proposition}
\label{prizeonly} Prize Only Competition: Let $\overline{\lambda }=\max_{i\in K}n_{i}p_{i}$. If party $\mathcal{A}$ offers a large prize ($\zeta_{A}>c$
and $\zeta_{B}=0$), there are no punishments ($\chi_{A}=0$, $\chi_{B}=0$), voting is costly ($c>0$) and there are no policy
differences ($G(x)=0$ for $x<0$\ and $G(x)=1$ for $x\geq 0$), then there
exist Nash equilibria in which at least two groups actively support $%
\mathcal{A}$. Further, if $j\in W_{A}$, then, either $p_{j}n_{j}=\overline{\lambda }$ or $p_{j}=1$. For \textit{non-}$\mathcal{A}$-active
groups, $i\notin W_{A}$, $p_{i}=0$ and $\prod\limits_{j\in
W_{A}}F_{n_{j}p_{j}}(1)\leq \frac{c}{\zeta_{A}}$.
\end{proposition}

\begin{proof}
From theorem ~\ref{existence}, equilibria exist. Further, from Myerson
(1998), we know all equilibria in random population games are type symmetric
so we restrict attention to such strategies. From equation ~\ref{AvsAbstain}
and with no policy differences and no punishments, equilibria require $PP_{A,k}\zeta_{A}=c$. We now
examine a series of cases:

1) Suppose no-one votes: $p_{j}=0$ for all groups $(j=1,\dots ,K)$. Then by
voting any voter could ensure that her group wins the prize. Since $\zeta _{A}>c$, someone wants to vote, so provided the prize is
larger than the cost of voting we can rule out $p_{j}=0$ for all groups.

2) Suppose $p_{j}>0$ for only one group. A member of this group is pivotal
if and only if no one else in her group votes. Hence $p_j>0$ implies $PP_{A,j}=\Pr
(A_{j}=0)=e^{-(n_{j}p_{j})}$. Consider a voter in group $i$, where $%
p_{i}=0$, a vote for party $\mathcal{A}$ wins a prize for group $i$ if $A_{j}\leq 1$.
Therefore for group $i$, $PP_{A,i}=\Pr (A_{j}=0)+\Pr
(A_{j}=1)>e^{-(n_{j}p_{j})}=PP_{A,j}$ (remember that a group receives a prize if it ties for highest support) which contradicts $p_{i}=0$.

3) We now consider the case where at least two groups vote with positive
probability: $p_{j},p_{k}>0$. If $p_{j}\in (0,1)$, then $PP_{A,j} \zeta_{A}=c$. If $%
p_{k}\in (0,1)$, then $PP_{A,k} \zeta_{A} =c$. Hence $PP_{A,j}=PP_{A,k}$ and proposition~\ref{useful_lemma}
implies $n_{j}p_{j}=n_{k}p_{k}$. If $p_{k}=1$ then $PP_{A,k} \zeta_{A} \geq c$.

Therefore, in groups that provide support for party $\mathcal{A}$, either
all $\mathcal{A}$-active groups generate the same number of expected votes
or all members of a group vote for $\mathcal{A}$.

4) Suppose no one in group $i$ votes ($p_{i}=0$), then the chance of being
prize pivotal is less than or equal to the cost of voting. Specifically, if $%
p_{j},p_{k}>0$ then $p_{i}=0$ implies $PP_{A,i} \zeta_{A}=\zeta_{A} \prod\limits_{j\in
K}F_{n_{j}p_{j}}(1) \leq c$. This last condition ensures that no one in group
$i$ wants to vote.
\end{proof}

In non-competitive elections with prize-only equilibria in which only party $%
\mathcal{A}$ offers a prize, we are essentially in a one-party environment
and, therefore of course, only party $\mathcal{A}$ receives any votes.
Despite there being only one credible party in this limiting case, the
extant groups divide into two sets; those that actively support $\mathcal{A}$
($W_{A}$) and those that provide no support ($K\backslash W_{A}$). In
general, the $\mathcal{A}$-active groups generate the same expected number
of voters for $\mathcal{A}$ and it is straightforward to see that $%
n_{i}p_{i} $ increases as the size of the prize $\zeta $ increases and
decreases in the cost of voting, $c$. So, broadly speaking, politicians in
this prize-only setting can shape the turnout rate -- or their mandate -- by
varying the size of the prize. Put differently, a budget constraint imposed
on the prize's size places a limit on turnout. The exception to this
group-turnout symmetry arises when the expected size of an active group is
less than the expected number of votes from other $\mathcal{A}$-active
groups. This situation can arise in equilibrium and when it does every voter
in the smaller group votes for $\mathcal{A}$. Contingent prize allocation
aligns the incentives of voters within groups and coordinates their actions.
Either many voters in a group vote for $\mathcal{A}$ or none of them vote.
The CPAR creates a competition between the groups that is supported by
individually rational voting. It also creates an incentive for a dominant
party to sustain more than one group or faction within its ranks.

Next we examine equilibria in a non-competitive electoral setting when groups compete to avoid penalties. 

\begin{proposition}\label{prop:penaltycompetition}
Penalty Only Competition: Let $\overline{\lambda }=\max_{i\in K}n_{i}p_{i}$. If party $\mathcal{A}$ punishes the uniquely least supportive group ($\chi_{A}>0$, $\chi_{B}=0$), there are no prizes ($\zeta_{A}=0$
and $\zeta_{B}=0$), voting is costly ($c>0$) and there are no policy
differences ($G(x)=0$ for $x<0$\ and $G(x)=1$ for $x\geq 0$), then
either (i) no groups support A ($p_{i}=0$, for all $i\in K$) or
(ii) all groups support A ($p_{i}>0$ for all $i\in K$), and either $p_{j}n_{j}=\overline{\lambda }$ 
or ($n_{j}<\overline{\lambda }$ and $p_{j}=1$). 
\end{proposition}
\begin{proof}
The vote calculus from equation ~\ref{AvsAbstain} implies that for group $i$ if $p_{i}\in (0,1)$ then $-\chi
_{A}QP_{A,i}=c$ and if $p_{i}=1$ then $-\chi _{A}QP_{A,i}\geq c$. 

We start with the pathological case. Suppose that there is some group $j$ such
that $p_{j}=0$. Since group $j$ delivers no votes for $\mathcal{A}$ and $\mathcal{A}$ only punishes
a group if it is the unique, least supportive group, $QP_{A,i}=0$ for all $i$.
Since voting is costly then no group votes: $p_{i}=0$ for all $i\in K$. 

Suppose $p_{i}>0$ for all $i\in K$. If $p_{j}\in (0,1)$ and $p_{k}\in (0,1)$
then $QP_{A,j}=QP_{A,k}=-\frac{c}{\chi _{A}}$. Therefore, from proposition
\ref{prop:deltaQP}, $\lambda _{j}=\lambda _{k}$. Let $\overline{\lambda }%
=\max_{i\in K}\lambda _{i}$, be the maximum expected turnout by any group.
From  proposition \ref{prop:deltaQP}, $|QP_{A,j}|$ is decreasing in $\lambda _{j}$,
so if $\lambda _{j}<\overline{\lambda }$  either group $j$ has a greater
incentive to vote than members of other groups or everyone in group $j$ votes for A ($n_{j}<\overline{\lambda }$). 
\end{proof}
    
The first case described in proposition \ref{prop:penaltycompetition} is clearly a pathology of the tie breaking rule that only a uniquely least supportive group is punished. In the second case, all groups compete to avoid the penalty that $\mathcal{A}$ allocates and all groups provide $\mathcal{A}$ the same level of support (unless this support level exhausts the number of available voters).  

Proposition \ref{prizeonly} tells us that prize competition involves at least 2 $\mathcal{A}$-active groups. However, as seen in proposition \ref{useful_lemma} and illustrated in figure  \ref{FigAsym.pdf}, equilibria with more than two such groups are knife-edged cases that are dynamically unstable. 
 With three groups, initial asymmetries
incentivize members of the lowest turnout group to turn out even less and
the largest turnout group to turn out even more. Obviously, such positive
feedback makes equilibria impossible except in the perfectly balanced
symmetric case. Analogous to Duverger's result that two-party competition
evolves under majoritarian rule (\citealt{duverger1959political}; \citealt{riker1982two}), WTA Contingent
Prize Allocation Rules result in stable competition between two groups
within any individual party. When competition is for prizes, then parties are supported by two groups of voters. Any smaller third group has little prospect of matching the votes of the larger groups and so has lower pivotality. Given this lower prize pivotality, the smaller group has a lower incentive to vote.   

An important distinction between awarding prizes and doling out penalties becomes critical in political settings with more than two potentially viable political groups. Contingent penalties induce stable turnout in multiple group settings. As we see in proposition \ref{prop:deltaQP}, the penalty pivot has an underdog effect that induces the group most likely to be punished to work harder to avoid the penalty. The contingent prize does not have this effect with more than two groups. Hence, politicians confronted by multiple groups have incentives to use punishments (perhaps in conjunction with prizes) to simultaneously induce support from within all groups. Punishments, then, alter Duvergerian expectations of competition being stable only between two groups. 

\section{Equilibrium Voting Behavior}

The competitiveness of elections plays an important role in distinguishing
expected voting behavior in the model. Informally, by non-competitive
elections we mean an election in which one party is virtually certain to
win. In terms of the model, this implies that $OP_{A}\rightarrow 0$. In this section we illustrate symmetric voting
behavior in competitive and non-competitive settings, after which we examine
asymmetric voting behavior between groups that leads to polarization in
which some groups predominantly back party $\mathcal{A}$, others back party $%
\mathcal{B}$ and others remain non-aligned.

Contingent Prize Allocation Rules induce turnout in both competitive and
noncompetitive elections. We illustrate the impact of prize and outcome pivots. First consider a completely symmetric situations in which there are 4 groups, each of the same expected size, and suppose that group 1 and 2 support $\mathcal{A}$ and groups 3 and 4 support $\mathcal{B}$. Consider the equilibrium in which each group votes for the candidate at the same rate: hence, $\lambda_1=\lambda_2=\gamma_1=\gamma_2$ and each party is equally likely to win the election.  Under such a symmetric circumstance, the ratio of prize pivot to outcome pivot is $\sqrt{2}$.

The relative importance of prizes compared to policy preferences has two significant implications for the \citet{black1948rationale} and \citet{downs:1957} views of party competition. First, if candidates converge to the median voter position, as predicted in Downsian competition, then policy
differences are of course irrelevant and prizes alone dictate voting
behavior. Second, because the impact of prizes can be so much larger than
the impact of policy differences, there may be little incentive for parties
to converge on policy.

In contrast to the standard rational voter story, elections need not be
close to induce turnout. Indeed, support for party $\mathcal{A}$ would be
relatively unchanged even if few people supported party $\mathcal{B}$.
Suppose for instance that party $\mathcal{B}$ had a far smaller prize to
offer than party $\mathcal{A}$, as might occur if $\mathcal{A}$ were a
long-term incumbent and $\mathcal{B}$ had never held office and was not
expected to do so in the near future. Cases of dominant parties are the norm
in many nations around the world \citep{kitschelt2007patrons}. Here we
relax the limiting conditions in proposition~\ref{prizeonly} in which $\zeta
_{B}=0$ so that we induce votes for each party. If parties $\mathcal{A}$ and
$\mathcal{B}$ offer prizes worth $\zeta _{A}$ and $\zeta _{B}$,
respectively, and all groups vote symmetrically, then A obtains
approximately $(\zeta _{A}/\zeta _{B})^{2}$ times the votes that B receives.
If $\zeta _{A}$ $>$ $\zeta _{B}$, then the assumption of non-competitive
elections is well justified since the probability that party $\mathcal{B}$
wins is trivial. For instance, if $\zeta _{A}=16$ and $\zeta _{B}=4$, then $%
OP_{A}\approx\frac{1}{\sqrt{\pi }\sqrt{\lambda }}e^{-4\lambda ^{2}}$ where $%
\lambda $ is the expected number of supporters for party $\mathcal{B}$. In
the case of 2 groups and a cost of voting, $\allowbreak c=\frac{1}{5\sqrt{%
\pi }}$, then, with $1000$ expected votes for party $\mathcal{B}$, the
chance of any voter being outcome pivotal is about $\frac{1}{160}\frac{\sqrt{%
10}}{\sqrt{\pi }}e^{-9000}$. In such equilibria, elections are free and
fair, but challengers can never expect to win. With little prospect of
attaining office and hence the resources with which to distribute prizes,
challengers garner relatively little support. Reform is hard. Even if all
the voters want political change, they want such changes enacted by voters
in other groups rather than risk their group's access to prizes by
diminishing their group's vote share. Once incumbents are expected to
win continually, such expectations become self fulfilling. 

\subsection{Polarization and Asymmetric Equilibrium Behavior}

Groups need not behave symmetrically, although as the examples above
illustrated they may. Returning to the nomenclature introduced earlier, $%
W_{A}$ refers to $\mathcal{A}$-active groups, those whose voters supported
party $\mathcal{A}$ with positive probability. These groups compete for the
prize offered by party $\mathcal{A}$. To avoid the need for a more general
definition of active groups, for this section we assume the distribution of
preferences over the parties, $G(x)$, is the uniform distribution over the
interval -1/2 to 1/2. Let $W_{B}$ refer to the set of groups that actively
support party $\mathcal{B}$. Although these assumptions are relaxed in the
next section, we assume here that groups are symmetric in size and
preferences. The proposition below generalizes proposition~\ref{prizeonly}.

\begin{proposition}
\label{polarization}For suitably sized prizes (and no penalties), there exist equilibria in
which there are at least two groups in $W_{A}$, at least two groups in $%
W_{B} $ and members of remaining groups vote for neither party. For interior
solutions, the vote probabilities satisfy the following equations:

\begin{equation}
\text{For }i\in W_{A}:OP_{A,i}(\frac{1}{2}-p_{i})+PP_{A,i}\zeta_{A}=c
\label{polarizeA}
\end{equation}

\begin{equation}
\text{For }j\in W_{B}:OP_{B,j}(q_{j}+\frac{1}{2})+PP_{B,j}\zeta_{B}=c
\label{polarizeB}
\end{equation}

\begin{equation}
\text{For }k\notin W_{A},W_{B}:OP_{A,k}(1/2)+PP_{Ak}\zeta_{A}<c\text{  and  }%
OP_{B,k}(1/2)+PP_{B,k}\zeta_{B}<c  \label{polarizeneutral}
\end{equation}
\end{proposition}

\begin{proof}
Given the uniform distribution, if $p_{i}\in (0,1)$ then $\tau _{i}=\frac{1}{%
2}-p_{i}$. Similarly, if $q_{i}\in (0,1)$ then $\tau _{i}=q_{i}+\frac{1}{2}$%
. Equations \ref{polarizeA} and \ref{polarizeB} are restatements of
equilibrium equations \ref{AvsAbstain} and \ref{BvsAbstain}. Ensuring $%
p_{i}\in (0,1)$ and $q_{i}\in (0,1)$ places bounds on prize size. The
constraint that non-aligned groups do not vote also places a constraint on
prize size. Since expected turnout for these groups is zero, a vote by a
member of such groups is only pivotal in the allocation of the prize if the
turnout from all other groups is 0 or 1: Hence $PP_{A,i\notin W_{A}}=\Pr (A_{1},\dots ,A_{K}\leq 1)=\prod\limits_{k=1,\dots
,K}F_{\lambda _{k}}(1)=\prod\limits_{k\in W_{A}}F_{\lambda _{k}}(1)
$. Applied to equation~\ref{polarizeneutral} this places an upper bound on
prize size.

The proof that there must be at least two active groups follows steps 1 and
2 in the proof of proposition \ref{prizeonly}.
\end{proof}

Proposition \ref{polarization} shows that even when outcome pivot
considerations are taken into account, equilibria retain many of the
properties seen in the limiting prize only case. In particular, provided
prizes are larger than the cost of voting, then voters from at least two
groups actively vote for each party. The $\mathcal{A}$-active and $\mathcal{B%
}$-active groups can be either disjoint, such that certain groups support
only party $\mathcal{A}$ while other groups support only party $\mathcal{B}$%
, or there could be overlap such that there are certain groups which have members who actively support party $\mathcal{A}$ and members who actively party $\mathcal{B}$. More than two active groups for either party can only
be supported under conditions of perfect balance (proposition~\ref{prizeonly}%
) or when penalties are used. Thus, except for perfect balance, there are not less than 2 and not more
than 4 stable groups within a party-office competition based on prizes. These groups need not constitute a majority within the polity. CPAR provide an explanation as to why groups pander to minorities: securing high turnout from two small groups may generate more overall support than appealing to the whole polity \citep{catalinac2015, myerson1993incentives}.

Two groups per party
supports the Duvergerian view of winner-take-all settings. The results,
however, add nuance to the Duvergerian perspective because with two parties
there can be up to four groups within each competition for prizes and, as we saw earlier, many more if penalties supplement prize incentives.
Furthermore, we will show later that parties can break competition for
office into many smaller prize competitions -- as in precinct votes in a
single Congressional district or electoral college votes across states in a
single presidential election. Then we will see that while no more than four
groups can be stable within any of the (prize only) sub-competitions, many more than four
groups can be supported in the overall party competition.

For convenience, the equilibria were stated for the uniform distribution and
the definition of active groups involved any positive probability of voting
for a party. However, we might imagine modifying these definitions.
Intuitively, members of $\mathcal{A}$-active groups (group $j$ for instance) have the prospect of
affecting both which party wins ($OP$) and the distribution of prizes ($%
PP_{A,j}$). In expectation, the aggregation of these individual incentives
delivers $\lambda _{j}$ expected votes for $\mathcal{A}$. In contrast,
non-aligned groups generate far fewer votes, zero in the case above.
However, if the uniform distribution assumption is relaxed and there is full
support over the preference types, then a strong party advocate ($%
|\varepsilon |$ very large) in non-aligned group $k$ wants to vote in order to
influence the outcome of the election. $\mathcal{A}$-active groups can be
redefined as groups with levels of support for $\mathcal{A}$ above some
substantial threshold to accommodate such a generalization. By voting for $%
\mathcal{A}$, a strong party advocate in a non-aligned group could
potentially win the prize for her group. However, since her group generates
relatively few votes for $\mathcal{A}$, her prospects of being prize pivotal
become vanishingly small.
Returning to the definition of prize pivot and supposing that the expected
number of votes from $\mathcal{A}$-active groups, $\lambda _{j}$, vastly
exceeds the expected number of votes in non-$\mathcal{A}$-active groups, $%
\lambda _{k}$, then

\begin{equation*}
PP_{A,k}= \sum_{a=0}^{\infty }f_{\lambda _{k}}(a)\left(
\prod\limits_{j\in W_{A}}F_{\lambda _{j}}(a+1)-\prod\limits_{j\in
W_{A}}F_{\lambda _{j}}(a)\right) \sim 0
\end{equation*}

Since $\lambda _{j}>\lambda _{k}$, for those $a$ where $f_{\lambda _{k}}(a)$
is of substantial magnitude $\prod\limits_{j\in W_{A}}F_{\lambda _{j}}(a+1)$
and $\prod\limits_{j\in W_{A}}F_{\lambda _{j}}(a)$ are approximately zero.
And when $\prod\limits_{j\in W_{A}}F_{\lambda _{j}}(a+1)$ $%
-\prod\limits_{j\in W_{A}}F_{\lambda _{j}}(a)$ is substantial (at values of $%
a$ around $\lambda _{j}$), $f_{\lambda _{k}}(a)$ is tiny. Hence, for members
of groups whose expected support for $A$ is substantially less than the
expected support from other groups, the prize pivot is very small. The
incentive for members of such groups is primarily to influence the electoral
outcome. The voting calculus in such groups is approximately $OP_{A}(\tau
_{A,k})\approx c$. In contrast voters in $\mathcal{A}$-active groups are
motivated by both outcome and prize considerations: $OP_{A}(\tau
_{A,j})+PP_{A,j}\zeta_{A}\approx c$. Given expectations about whether or not their group
has a realistic prospect of being awarded the prize, voter incentives
fulfill such expectations.

Thusfar we have explored symmetric and asymmetric equilibrium behavior under
conditions in which group sizes are symmetric. Now we examine the
implications when group sizes are asymmetric.

\section{Asymmetry and Group Formation}

When elections are competitive, asymmetries in group size and in preferences
are important. Although in the symmetric case, knife edged equilibria
involving more that two $\mathcal{A}$-active groups exist, in the presence
of asymmetry in group size or in group party preferences, parties draw
the bulk of their support from just two groups when prizes are used contingently to incentivize voters.

We start by examining how asymmetries between groups affect the stability of
equilibria under competitive elections. Then we examine group formation and
dynamics from the perspective of voters and political entrepreneurs and
parties. 

\subsection{Asymmetry and Instability}

We now examine the consequences of structural asymmetries between groups
rather than their behavioral differences. The results are
phrased in terms of asymmetries with respect to group size, but shifts in
group preferences for one party over the other have similar effects.
 
A useful starting point for exploring asymmetry is the equilibrium
conditions for $\mathcal{A}$-active groups 1 and 2 that govern indifference
between voting for $A$ and abstaining. From equation~\ref{AvsAbstain}

\begin{equation}
OP_{A}\tau_{A,1}+PP_{A,1}=c=OP_{A}\tau_{A,2}+PP_{A,2}  \label{balance1}
\end{equation}

Rearranging this equation, substituting $\tau_{A,1}=G^{-1}(1-p_{1})$ and
noting that while the prize pivot varies by group, the outcome pivot does
not, yields:

\begin{equation}
OP_{A}(G^{-1}(1-p_{1})-G^{-1}(1-p_{2}))=PP_{A,2}-PP_{A,1}  \label{balance2}
\end{equation}

In the case of the uniform distribution analyzed above, the $%
G^{-1}(1-p_{1})-G^{-1}(1-p_{2})$ term is simply $p_{2}-p_{1}$. For group 2
to support $\mathcal{A}$ at a higher rate than group 1 ($p_{2}>p_{1}$)
implies that $PP_{A,2}>PP_{A,1}$. Combining this with proposition~\ref%
{useful_lemma}, leads directly to the following result:

\begin{proposition}
\label{2group}In competitive elections ($OP_{A}>0$), if there are two $%
\mathcal{A}$-active groups and $n_{1}>n_{2}$, then $p_{2}>p_{1}$ and $%
n_{1}p_{1}>n_{2}p_{2}$. However, as $OP_{A}\rightarrow 0$, $%
n_{1}p_{1}\rightarrow n_{2}p_{2}$.
\end{proposition}

This result implies that with competitive elections more individuals from
the larger group turn out to support $\mathcal{A}$ even though a larger
percentage of the smaller group supports $\mathcal{A}$ -- the underdog effect discussed earlier. Intuitively, for
both groups to provide the same level of support for $\mathcal{A}$ requires
that a higher proportion of the smaller group votes. However this implies
that the indifferent type in the smaller group ($\tau _{A,2}$) likes $%
\mathcal{A}$ less than the indifferent type in the larger group ($\tau
_{A,1}$) and is therefore less motivated to turnout. In non-competitive
elections this preference distinction between groups is of little
consequence. When the electoral winner is not in doubt, policy preferences
over parties do not enter voters' considerations. In the competitive
election setting, the larger group generates a higher expected level of
support for $\mathcal{A}$ and so on average wins the prize. This differential expectation of
winning prizes shapes voter incentives to migrate from one group to another.

\subsection{Voter Organized Groups}

The competition for prizes affects which social cleavages are active and the
evolution of group identity. The competitiveness of elections and the extent
to which prizes are rival or non-rival effect the incentives of voters to
migrate between groups. If a voter were offered the choice to switch groups
prior to playing the voting game, then her propensity to do so depends on
several factors. First, there is an innate personal cost to switching group
identity. Such an emigration cost depends upon the nature of groups. If
groups are geographically based, then the cost is that of relocating. Other
emigration costs might be less tangible, such as learning a new language or
religious practice. Second, beyond the personal cost of emigration, groups
differ in the extent to which they welcome members. Extant members can
charge a high immigration cost for people wishing to join. Alternatively,
they might actively seek to redefine their group's identity to be more
inclusive. The willingness of people to pay the costs of migration and the
barriers that groups set to entry depend on the level of political
competition, the size of prizes and the rival/non-rival nature of the prizes.

When elections are non-competitive the rate of migration between groups is
low and group identities are static. Since the outcome of the election is a
foregone conclusion, policy preferences are irrelevant considerations and
so, beyond needing enough members, precise group size has little effect on
equilibrium turnout. Voters gain little from migration as each $\mathcal{A}$%
-active group has the same equilibrium probability of winning the prize and,
if the prizes are rival, then extant members of groups want to restrict
entry because additional members dilute their share of the prizes without
increasing the likelihood that the group wins a prize. Thus, in the
non-competitive electoral setting, with groups incentivized to raise the
cost of immigration for potential migrants and migrants having little to
gain from migration, group identities are relatively fixed, especially if
prizes are rival.

Group dynamics are more fluid under competitive electoral settings and it is under such circumstances that \citet{chandra2007ethnic} observes that political entrepreneurs attempt to redefine group identities for electoral gain. Unlike the non-competitive setting where expected size differences have little impact on which group wins the prize, when elections are competitive the
larger group has higher expected turnout relative to the smaller group. As
the difference in group size grows, the larger group (group 1) becomes
increasingly likely to be awarded the prize compared to the smaller group
(group 2). This has important implications for group dynamics and electoral
competitiveness.

In the competitive setting, both groups 1 and 2 have incentives to absorb
additional members from non-$\mathcal{A}$-active groups (or from each
other). By taking on additional members, each group increases the
probability of attaining the prize. However, this benefit has to be
contrasted against the number of members who share a rival prize. The desire
to increase the likelihood of winning prizes induces an openness on the part
of groups (lowering immigration costs) and the prospects of winning prizes
creates an incentive for individuals to migrate to $\mathcal{A}$-active
groups. However, the incentive to welcome immigrants into the group
evaporates when the dilution of the value of the prize exceeds the marginal
improvement in the probability of winning the prize. Groups are more open to
absorbing new members when prizes are non-rival.

The fluidity of group membership has the potential to undermine the
competitiveness of elections when one group is more successful at recruiting
members than another. As group 1 grows in size relative to group 2, both
groups reduce their support for party $\mathcal{A}$ because prize pivotality
declines as group 1 becomes more likely to be awarded the prize, as shown in
proposition~\ref{2group}. As the size difference between groups 1 and 2
increases, the number of votes for party A declines in both groups and so $%
\mathcal{A}$ loses more elections. Such a reduction in electoral
competitiveness reduces the incentive to migrate and dampens disparities
between groups. As elections become non-competitive, the incentive for
voters to form larger groups vanishes. The necessity that elections remain
competitive limits the extent to which one $\mathcal{A}$-active group can be
more successful than the other at recruiting new members.

The equilibrium-induced coordination among group members
imposes limits on how far apart groups drift in size when migration across
groups is possible. Politicians also have incentives to influence group
divisions. We now examine those incentives and how they may influence
electoral competition and turnout.

\subsection{How politicians organize groups}

Although many group identities, such as race, religion or ethnicity, may be
primordial, others clearly are artificial constructs created by politicians. For instance, the City of Chicago is divided into 50 wards.
Why did politicians create 50 groups, instead of 2, 20 or 200, and how is
political competitiveness structured between these groups? Organizing direct prize competition between 50 wards is unlikely to
engender high levels of political support because, in the presence of
asymmetry, only two $\mathcal{A}$-active groups (e.g. wards) are part of a stable
equilibrium in a competition for prizes. However, rather than organize a single competition between many
groups, parties can create numerous competitions between smaller subsets of
groups such as has been done in Chicago. When prizes are rival it is in
their interest to do so.

To illustrate how parties structure competition for prizes, suppose there are 16
roughly evenly sized groups, referred to in our illustration as wards. A dominant party has the goal of obtaining
the support of 60\% of the voters while minimizing its expenditure on
prizes. Given 60\% support, elections are non-competitive so we restrict
attention to prize motivations. 
 If party $\mathcal{A}$ offers a single prize
to the most supportive of the 16 wards, then in equilibrium only two wards 
will be highly supportive and the 60\% support goal can not be realized.
However, parties can structure competition differently through various
aggregation processes.  For instance, the administration of services for the 16 wards can be divided between two or more counties. If there are two counties, then the 16 wards are divided 8 and 8. 
A party might then award the prize to the county that generates
the greatest number of votes, in which case there is a single prize that is shared by the 8 wards in the most supportive county. Alternatively, the party might form 8 counties
each containing 2 wards and have a competition
for a prize in each of the eight counties. Which configuration is optimal depends
upon the rivalness of prizes.

Table 1 calculates the relative cost of prize provision needed to elicit a total of 
60\% support under different divisions of wards into counties. These costs depend
upon three factors: 1) the number of prizes, which depends upon the number
of counties, 2) the number of people benefiting from each prize and
whether the prize is rival or non-rival, and 3) the number of wards in each county, as this
effects prize pivotality. Table 1 illustrates the impact of each of these factors
for rival and non-rival prizes.

As political competition is broken into a series of small sub-contests,
parties must provide more prizes. The non-rival prize assumes that everyone in
the group benefits from the total value of the prize and the cost of prize provision is
unrelated to the size of the group. When prizes are
non-rival, parties prefer a small number of groups that compete within a
single competition for the prize. Once provided, all members of the group
enjoy the prize so creating multiple competitions simply means additional
expenditure because multiple prizes have to be created. Therefore, the theory implies that competition over
non-rival prizes, such as language or religious supremacy, takes place at
the national level. In contrast, when prizes are rival, such as traditional
patronage goods, then parties can reduce their expenditure by creating
numerous smaller competitions.

\noindent{Table 1: Competitions and the Overall Cost of Prize Provision.*}

{\small \noindent
\begin{tabular}{|l||c|c|c|c|}
\hline
Number of Counties & 2 & 4 & 8 & 16 \\  \hline
Number of Wards per County & 8 & 4 & 2 & 1 \\ \hline\hline
Relative Cost: Non-Rival Prize & {$\mathbf{\ 1\cdot 1\cdot\sqrt{8}=2.83}$ } & $2\cdot
1\cdot \sqrt{4}=4$ & $4\cdot 1 \cdot \sqrt{2}=5.66$ & $8\cdot 1\cdot \sqrt{1}%
=8$ \\ \hline
Relative Cost: Rival Prize & $1\cdot 8\cdot \sqrt{8}=22.63$ & $2\cdot 4\cdot \sqrt{4}=16$ &
$4\cdot 2\cdot \sqrt{2}=11.31$ & {$\mathbf{8\cdot 1\cdot \sqrt{1}=8}$ } \\
\hline
\end{tabular}

\noindent 
* In each cell of the last two rows, the first value corresponds to the number of prizes, the second refers to the cost of generating the prize (which increases with the number of wards for rival prizes) and the third relates to each county's size (and hence pivotality).
}

In the rival prize setting, whether there is a single competition between 2 counties each composed of 8 wards or 8 smaller county competitions, half the people are in wards that
receive prizes. Parties can generate the same level of support by breaking
political contests into numerous smaller competitions. Prize pivots are on
the scale of $1/\sqrt{n}$. Thus a single competition requires a rival prize
of size $\zeta \approx c/PP_{A,i}=c2\sqrt{\pi }\sqrt{p\frac{n_{T}}{2}}$ \
where $p=.6$ indicates the 60\% support goal. In contrast, if political
competition is broken down into 8 separate competitions then obtaining 60\%
support requires rival prizes on the scale of $\zeta \approx c2\sqrt{\pi }%
\sqrt{p\frac{n_{T}}{16}}$ which means the size of the rival prize each
recipient needs to receive is only $\frac{1}{\sqrt{8}}$ of the value
required in the single competition case. By breaking the competition for
rival prizes into a series of local competitions, parties can reduce the
amount of resources they need to elicit political support. 

\section{Conclusions}

Using a Poisson games framework of Myerson (1998, 2000), we have modeled
elections in which parties offer prizes or penalties to identifiable
groups of voters on a contingent basis. The
model demonstrates that even in large populations, in which voters have
little influence on the outcome of elections, they retain significant
influence over the distribution of prizes. This influence persists even in
lopsided elections, giving all political parties, whether in competitive or
non-competitive environments, an incentive to encourage factions so as to
manage turnout and achieve the appearance of a mandate whether they are
popular or not. We have specified the conditions under which voter turnout
fluctuates as a function of four considerations: the value of contingent
prizes or penalties; the extent to which prizes are rival or non-rival; the degree to
which elections are competitive; and the extent to which the size of voter
groups are symmetric or asymmetric. Equilibrium behavior is more likely to
be driven by voters competing to win preferential treatment for their group
than by policy concerns. The model also provides a modified and more nuanced
understanding of the implications of winner-take-all settings beyond the
standard Duvergerian account. When penalties rather than prizes are used turnout can be induced and group stability sustained for any number of groups. 

The results offer insights into group dynamics. In non-competitive
electoral settings, voters have little reason to shift groups or alter their
political identity and group members have little reason to welcome the
entreaties of others to join them. In contrast, competitive elections induce
fluidity in group membership, at least up to a limit. When prizes are rival,
then extant groups of voters are happy to welcome new members as long as
they improve the probability of winning the prize more than they dilute the
value of the prize. In this way we can see both fluidity and self-sustaining
features to prize-motivated groups and explanations for variance in turnout
by rational voters. 
\bibliographystyle{apsr}

\bibliography{poisson}

\newpage

\begin{figure}
\begin{center}
\includegraphics[height=3.7in]{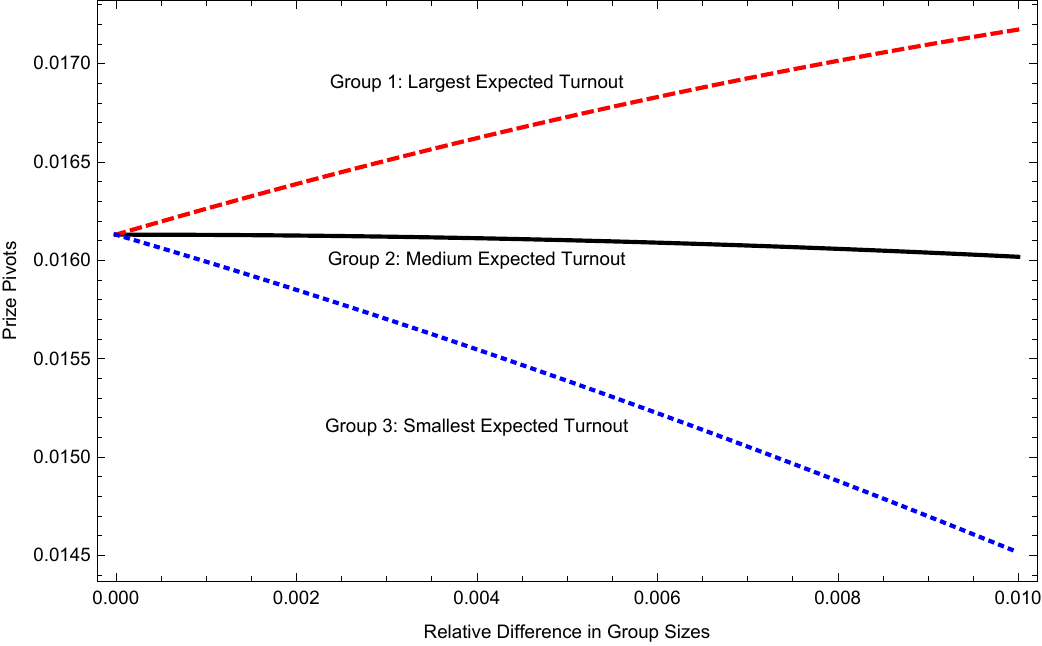}
\caption{Prize Pivots and Asymmetric Group Size}
\label{FigAsym.pdf}
\end{center}
\end{figure}

\end{doublespacing}
\end{document}


\title{Online Appendix: Group Incentives and Rational Voting}
\author{Alastair Smith \and Bruce Bueno de Mesquita  \and Tom LaGatta}

\maketitle

\pagenumbering{arabic}

{\normalsize 
This appendix examines two issues. First, we show the positive feedback in prize pivots when 3 or more groups actively support $\mathcal{A}$. Second, we derive asymptotic approximations for prize pivots in the case of more than two $\mathcal{A}$-active groups.

\subsection{Positive Feedback in Prize Pivots with more the 2 $\mathcal{A}$ Groups}
In the main text we simply illustrated the positive feedback in prize pivots from increased group turnout with a graph. Here we address the issue more systematically. 
As equation 6 shows, $\Delta P$ contains three terms. The first and third
terms are summations over two Probability Mass Functions and a CDF, while the second term involves
three PMF. Provided the $\lambda $'s are large the second term is
small in magnitude compared to the other two terms. We define $\widehat{%
\Delta P}$ as the first and third terms of equation 6. For three groups
with turnout $\lambda _{1},\lambda _{2}$ and $\lambda _{3}$ and noting that $%
f_{\lambda }(a+1)=f_{\lambda }(a+1)\frac{\lambda }{a+1}$ and $F_{\lambda
}(a)=\sum_{x=0}^{a}f_{\lambda }(a)$, we can write

\begin{equation}
\widehat{\Delta P}=\sum_{a=0}^{\infty }\frac{e^{-\lambda _{1}-\lambda
_{2}-\lambda _{3}}}{a!a!(a+1)}\sum_{x=0}^{a}\left( \frac{\lambda
_{1}^{a}\lambda _{2}^{a}\lambda _{3}^{x}}{x!}(\lambda _{2}-\lambda _{1})+%
\frac{\lambda _{3}^{a}\lambda _{3}}{x!}(\lambda _{1}^{a}\lambda
_{2}^{x}-\lambda _{2}^{a}\lambda _{1}^{x})\right) 
\end{equation}
Presented in these terms, the second term of $\Delta P$ is  
\begin{equation}
\sum_{a=0}^{\infty }\frac{e^{-\lambda _{1}-\lambda
_{2}-\lambda _{3}}}{a!a!(a+1)}\frac{\lambda_{1}^{a} \lambda_{2}^{a} \lambda_{3}^{a}}{a!(a+1)}\lambda_{3}(\lambda_{2}-\lambda_{1})
\end{equation}
Suppose that $\lambda _{1}=\lambda \rho $ and $\lambda _{2}=\lambda
_{3}=\lambda $, $\widehat{\Delta P}=\sum_{a=0}^{\infty }\frac{e^{-2\lambda -\lambda \rho
}\lambda ^{a}\lambda ^{a}\lambda }{a!a!(a+1)}S(a)$, where $%
S(a)=\sum_{x=0}^{a}\frac{\lambda ^{x}}{x!}\left( \rho ^{a}(1-\rho )+\rho
^{a}-\rho ^{x}\right) $.
We now examine how $\widehat{\Delta P}$ varies with $\rho $ evaluate at $%
\rho =1$.
\begin{equation}
\frac{d\widehat{\Delta P}}{d\rho }=\sum_{a=0}^{\infty }-\lambda \frac{%
e^{-2\lambda -\lambda \rho }\lambda ^{a}\lambda ^{a}\lambda }{a!a!(a+1)}%
S(a)+\sum_{a=0}^{\infty }\frac{e^{-2\lambda -\lambda \rho }\lambda
^{a}\lambda ^{a}\lambda }{a!a!(a+1)}\frac{dS(a)}{d\rho } \end{equation}
where \begin{equation} \frac{dS(a)%
}{d\rho }=\sum_{x=0}^{a}\frac{\lambda ^{x}}{x!}\left( 2a\rho
^{a-1}-(a+1)\rho ^{a}-x\rho ^{x-1}\right)  \end{equation}

Evaluated at $\rho =1$, $S(a)=0$ and $\frac{dS(a)}{d\rho }=\sum_{x=0}^{a}%
\frac{\lambda ^{x}}{x!}\left( a-1-x\right) =\frac{\lambda ^{a+1}}{a!}+\frac{%
e^{\lambda }(a-1-\lambda )\Gamma (a+1,\lambda )}{a!}$. The only negative
term in the summation for $\frac{dS(a)}{d\rho }$ corresponds to $x=a$ and all other terms in the
sequence are positive. Hence the summation is increasing $a$. To see this formally, note that $e^{-\lambda }\sum_{x=0}^{a}\frac{\lambda ^{x}}{%
x!}\left( a-1-x\right) <e^{-\lambda }(-\frac{\lambda ^{a}}{a!}%
+\sum_{x=0}^{a-2}\frac{\lambda ^{x}}{x!})=F_{\lambda }(a-2)-f_{\lambda }(a)$. Hence $F_{\lambda }(a-2)\geq f_{\lambda
}(a)$ is a sufficient condition for $\sum_{x=0}^{a}%
\frac{\lambda ^{x}}{x!}\left( a-1-x\right) >0$ and $F_{\lambda }(a-2)- f_{\lambda
}(a)$ is strictly increasing in $a$. Further for large $\lambda$,  when $F_{\lambda }(a-2)- f_{\lambda
}(a)$ is negative (which occurs for only for $a<<\lambda$), it is small in magnitude.

$\left. \frac{d\widehat{\Delta P}}{d\rho }\right\vert _{\rho
=1}>\sum_{a=0}^{\infty }f_{\lambda }(a)^{2}\frac{\lambda}{a+1}(F_{\lambda }(a-2)-f_{\lambda
}(a))$ which is positive for large $\lambda $. 
\begin{lemma}
For $\lambda _{1}=\lambda \rho $ and $\lambda _{2}=\lambda
_{3}=\lambda $ and $\lambda$ large, 
$\left. \frac{d\widehat{\Delta P}}{d\rho }\right\vert _{\rho=1}>0$
\end{lemma}

\subsection{Asymptotic approximation for the
prize pivot for the general case of $K$ $\mathcal{A}$-active groups.}
The proposition below provides an approximation for $K=3$ \ groups based upon
the expansion of equation 2, hence we refer to it as the
expansion approximation. The proof deals with the general case of $K$
groups. As $K$ increases, the expansion approximation involves $(K-1)+\frac{%
(K-1)(K-2)}{2}+\frac{(K-1)(K-2)(K-3)}{3!}+\dots $ terms. However as expected
turnouts become large the latter terms become small relative to the first $%
(K-1)$ terms.\footnote{%
The $(K-1)$ terms involve summations over $K-2$ CDF terms and 2 PMF terms.
The $(K-1)(K-2)/2$ terms involve summations over $K-3$ CDF terms and 3 PMF
terms. As the $\lambda $'s become large the PMF become small relative to the
CDF terms, so the later terms can be ignored.} As $K$ increases this results
in many terms, although in symmetric cases many of these terms are the same.
The approach yields accurate approximations as $np$ increases. For instance,
for $K=3$ and $\lambda _{1}=\lambda _{2}=\lambda _{3}=1000$, the error in
the expansion approximation is about 1.5\%. As expected turnouts increase,
the approximation become increasingly accurate. Let $n$ denote the average
population size, so that $n\omega _{i}$ denotes the expected size of group $i
$. Let $\lambda _{i}=n\omega _{i}p_{i}$ denote the expected number of votes
for party $\mathcal{A}$ from group $i$. }

{\normalsize To analyze pivot probabilities correctly, we distinguish
between different types of asymptotic notation. If $f = f(n)$ and $g = g(n)$
are two functions depending on $n$, we write $f \approx g$ to indicate that
their difference tends to zero as $n\to\infty$. That is, $f \approx g$ if
and only if $|f(n) - g(n)| \to 0$. This is a very strong type of
convergence, and is generally not suitable for situations where both $f$ and
$g$ tend to infinity, since usually their difference $|f-g|$ often also
tends to infinity, albeit at a much slower rate. }

{\normalsize Instead, we write $f\sim g$ to indicate that the ratio of the
two functions tends to one as $n\rightarrow \infty $. That is, $f\sim g$ if
and only if $f/g\approx 1$, meaning that $|\tfrac{f(n)}{g(n)}-1|\rightarrow 0
$. This weak form of convergence allows us to efficiently analyze pivot
probabilities when the mean population size is large since probabilities are
calculated by taking ratios of large quantities. }

{\normalsize The crucial assumption we make is that $\lambda _{i}\rightarrow
\infty $ as $n\rightarrow \infty $, corresponding to an equilibrium with
sufficiently large turnout. The rate at which $\lambda _{i}\rightarrow
\infty $ does not matter for our analysis; simply that the limit exists and
is infinite is sufficient. This is not an innocuous assumption, and in the
worst-case scenario it is not true. Without a local incentive structure,
the usual rational-choice theory implies that $p_{i}\sim c_{i}/n$, for some
constant $c_{i}$ depending on the district. This implies that $\lambda
_{i}\sim \omega _{i}c_{i}$ instead of $\lambda _{i}\rightarrow \infty $.
Historically, the central question of rational-choice voting models has been
to prove results of the form $p_{i}\approx \bar{p}_{i}$ for some constant $%
\bar{p}_{i}\in (0,1]$, which entails a total turnout of $\lambda _{i}\sim
\bar{p}_{i}\omega _{i}n$. }

{\normalsize In our analysis, by assuming reasonably sized incentives and by
assuming that $\lambda _{i}\rightarrow \infty $, we are able to guarantee
that $\lambda _{i}\sim \bar{p}_{i}\omega _{i}n$ for some constant $\bar{p}%
\in (0,1]$. Our prize assumption is simply that the winning group receives a
prize of $\zeta \sim \zeta ^{\prime }\sqrt{n}$ from party $\mathcal{A}$,
where $\zeta ^{\prime }$ is a positive constant not depending on $n$. If
prizes are much smaller than this, then $p_{i}\sim c_{i}/n$ as the classic
theory predicts. While our results do not completely rule out the
possibility that $\lambda _{i}\not\rightarrow \infty $ in some equilibria,
they render that prospect extremely unlikely. }

{\normalsize We now estimate the asymptotic behavior of the prize pivot as $%
n \to \infty$.}

\begin{proposition}
{\normalsize \label{pro_PPapprox} Winner-Take-All Prize Pivot Approximation:
}

{\normalsize Consider the case of $K=3$ groups, and define the functions
\begin{eqnarray*}
y_{2}(a) &=&-\omega _{1}p_{1}-\omega _{2}p_{2}+\frac{1}{n}\log \left( \frac{%
(n\omega _{1}p_{1})^{a}}{\Gamma (a+1)}\frac{(n\omega _{2}p_{2})^{a}}{\Gamma
(a+1)}\frac{(n\omega _{2}p_{2})}{(a+1)}\prod\limits_{j\neq 1,2}\Phi \Big(%
\frac{a-n\omega _{j}p_{j}}{\sqrt{n\omega _{j}p_{j}}}\Big)\right)  \\
y_{3}(a) &=&-\omega _{1}p_{1}-\omega _{3}p_{3}+\frac{1}{n}\log \left( \frac{%
(n\omega _{1}p_{1})^{a}}{\Gamma (a+1)}\frac{(n\omega _{3}p_{3})^{a}}{\Gamma
(a+1)}\frac{(n\omega _{3}p_{3})}{(a+1)}\prod\limits_{j\neq 1,3}\Phi \Big(%
\frac{a-n\omega _{j}p_{j}}{\sqrt{n\omega _{j}p_{j}}}\Big)\right)  \\
y_{23}(a) &=&-\omega _{1}p_{1}-\omega _{2}p_{2}-\omega _{3}p_{3}+\frac{1}{n}%
\log \left( \frac{(n\omega _{1}p_{1})^{a}}{\Gamma (a+1)}\frac{(n\omega
_{2}p_{2})^{a}}{\Gamma (a+1)}\frac{(n\omega _{2}p_{2})}{(a+1)}\frac{(n\omega
_{3}p_{3})^{a}}{\Gamma (a+1)}\frac{(n\omega _{3}p_{3})}{(a+1)}\right) ,
\end{eqnarray*}%
where $\Phi $ represents the standard normal distribution function, and $%
\Gamma $ is the gamma function. For each of the three indices $u=2,3,23$,
let $\alpha _{u}$ denote the unique critical point to the function $y_{u}$,
which corresponds to a global maximum.\footnote{%
i.e., $\alpha _{u}$ solves the equation $y_{u}^{\prime }(\alpha _{u})=0$.
Note that the functions $y_{u}$ and the critical points $\alpha _{u}$ all
depend on $n$.} }

Define the maximum value $y_{max}(n) := \max_{u = 2, 3, 23} y(\alpha_u)$,
and let $\mathcal{U}_n := \big\{ u : y(\alpha_u) = y_{ max} \big\}$ denote
those indices which attain the maximum. Define
\begin{equation}  \label{rhodef}
\rho(n) = \sum_{u \in \mathcal{U}_n} \sqrt{\frac{2\pi}{n | y^{\prime \prime
}_u(\alpha_u)|}}.
\end{equation}

{\normalsize Suppose that $\lambda_i \to \infty$ as $n \to \infty$ for at
least one district. Then
\begin{equation*}
PP_{A,1}\sim \widetilde{PP_{A,1}}(n) := \zeta \rho(n) e^{n y_{max}(n)}
\end{equation*}
}

{\normalsize as $n \to \infty$.}
\end{proposition}

\begin{proof}
{\normalsize First we derive the general $K$ group expansion approximation,
then specialize to the case $K=3$. From the definition of prize pivot we
expand the products: }

{\normalsize
\begin{eqnarray*}
PP_{A,1} &=&\zeta \sum_{a=0}^{\infty }f_{\lambda_{1}}(a)(\prod\limits_{j\neq
1}F_{\lambda_{j}}(a+1)-\prod\limits_{j\neq 1}F(a)) \\
&=&\zeta \sum_{a =0}^{\infty }f_{\lambda_{1}}(a)\prod\limits_{j\neq
1}(F_{\lambda_{j}}(a)+\frac{\lambda_{j}}{(a+1)}f_{\lambda_{j}}(a))-\zeta
\sum_{a =0}^{\infty }f_{\lambda_{1}}(a)\prod\limits_{j\neq
1}F_{\lambda_{j}}(a) \\
&=&\zeta \sum_{a =0}^{\infty }f_{\lambda_{1}}(a)[\sum_{i\neq 1}\frac{%
\lambda_{i}}{(a+1)}f_{\lambda_i}(a)\prod\limits_{j\neq i,1}F_{\lambda
_{j}}(a)+ \\
&&\sum_{i}\sum_{j}\frac{\lambda_{i}}{(a+1)}f_{\lambda_i}(a)\frac{\lambda _{j}%
}{(a+1)}f_{\lambda_{j}}(a)\prod\limits_{k\neq i,j,1}F_{\lambda
_{k}}(a)+\dots].
\end{eqnarray*}
}

{\normalsize To analyze this expression, we define
$s_{i}=\zeta \sum_{a
=0}^{\infty }f_{\lambda_{1}}(a)\frac{\lambda_{i}}{(a+1)}f_{\lambda
i}(a)\prod\limits_{j\neq i,1}F_{\lambda_{j}}(a)$ \\ and $s_{ij}=\zeta \sum_{a
=0}^{\infty }f_{\lambda_{1}}(a)\frac{\lambda_{i}}{(a+1)}f_{\lambda_i}(a)%
\frac{\lambda_{j}}{(a+1)}f_{\lambda _{j}}(a)\prod\limits_{k\neq
i,j,1}F_{\lambda_{k}}(a)$, etc., so that $PP_{A,1} = \sum_i s_i + \sum_{ij}
s_{ij} + \sum_{ijk} s_{ijk} + \dots$. If there are three groups, then $%
PP_{A,1}=s_{2}+s_{3}+s_{23}$. Note that as the $\lambda $'s \ become large,
the higher-order terms ($s_{ij}$ and $s_{ijk}$ \dots) become small relative
to $s_{i}$'s. }

{\normalsize We estimate each term as an integral given by the
Euler-Maclarin formula. If $h(a)$ is any integrable real-valued function,
the Euler-Maclarin formula states that
\begin{equation*}
\sum_{a=0}^{w}h(a)=\int_{1}^{w}h(a)da-B_{1}(h(w)-h(1))+\sum_{y=1}^{r}\frac{%
B_{2y}}{(2y)!}(h^{(2y-1)^{\prime }}(w)-h^{(2y-1)^{\prime }}(1))+R,
\end{equation*}
} {\normalsize where $B_{1}=-\frac{1}{2}$, $B_{2}=1/6$, $%
B_{3}=0,B_{4}=-1/30,B_{6}=1/42,$\dots are the Bernoulli numbers, and $R \to
0 $ as $w \to \infty$. These estimates are well justified since once the $%
\lambda $'s are greater than 100, the higher order derivatives evaluated at $%
0$ and $3\lambda $ are less than $10^{-50}$. }

{\normalsize We analyze the term $s_2$ in detail; the corresponding analyses
for the terms $s_3$ and $s_{23}$ are similar so we omit them. Applying the
Euler-Maclarin formula,
\begin{eqnarray*}
s_{2} \,\approx\, \widetilde{s_{2}} &:=& \zeta \int_{a=0}^{\infty
}f_{\lambda _{1}}(a)\frac{\lambda_{i}}{(a+1)}f_{\lambda_i}(a)\prod\limits_{j%
\neq i,1}F_{\lambda_{j}}(a)da \\
&=&\zeta \int_{a=0}^{\infty }e^{-\lambda_{1}}\frac{\lambda_{1}^{a}}{\Gamma
(a+1)}e^{-\lambda_{2}}\frac{\lambda_{2}^{a}}{\Gamma (a+1)}\frac{\lambda _{2}%
}{(a+1)}\prod\limits_{j\neq 1,2}F_{\lambda_{j}}(a)da
\end{eqnarray*}
}

{\normalsize Next we substitute $n_{i}=\omega_{i}n$ where $n$ is the average
group size. By the Central Limit Theorem, the Poisson distribution function
is well-approximated by the Gaussian distribution function: $F_{n\omega
_{j}p_{j}}(a)\approx \Phi (\frac{a-n\omega_{j}p_{j}}{\sqrt{n\omega_{j}p_{j}}}%
)$. Applying these, we have }

\begin{eqnarray*}
\widetilde{s_{2}} &=&\zeta \int_{a=0}^{\infty }e^{-n\omega
_{1}p_{1}}e^{-n\omega_{2}p_{2}}\frac{(n\omega_{1}p_{1})^{a}}{\Gamma (a+1)}%
\frac{(n\omega_{2}p_{2})^{a}}{\Gamma (a+1)}\frac{(n\omega_{2}p_{2})}{(a+1)}%
\prod\limits_{j\neq 1,2}\Phi (\frac{a-n\omega_{j}p_{j}}{\sqrt{n\omega
_{j}p_{j}}})da \\
&=&\zeta \int_{0}^{\infty }e^{n(y_{2}(a))}da,
\end{eqnarray*}

{\normalsize where
\begin{eqnarray*}
y_{2}(a) &=&-\omega_{1}p_{1}-\omega_{2}p_{2}+\frac{1}{n}\log (\frac{%
(n\omega_{1}p_{1})^{a}}{\Gamma (a+1)}\frac{(n\omega_{2}p_{2})^{a}}{\Gamma
(a+1)}\frac{(n\omega_{2}p_{2})}{(a+1)}\prod\limits_{j\neq 1,2}\Phi (\frac{%
a-n\omega_{j}p_{j}}{\sqrt{n\omega_{j}p_{j}}})) \\
&=&-\omega_{1}p_{1}-\omega_{2}p_{2}+\frac{1}{n}(a\log (n\omega
_{1}p_{1})+(a+1)\log (n\omega_{2}p_{2}) \\
&& ~~~ -\frac{2}{n} \log \Gamma (a+1) - \frac 1 n \log (a+1)+ \frac 1 n
\sum_{j\neq 1,2}\log \Phi \Big(\frac{a-n\omega_{j}p_{j}}{\sqrt{%
n\omega_{j}p_{j}}}\Big).
\end{eqnarray*}
}

{\normalsize Our methodology now is to estimate the integral $\widetilde{%
s_{2}}=\zeta \int_{0}^{\infty }e^{n(y_{2}(a))}da$ using a modification of
Laplace's method. The idea is that, since $n$ is large, nearly all the value
of this integral is concentrated near the maximum of the function $y_2$.
This is not difficult, but care must be taken since both the function $y_2$
and its critical point depend on the large parameter $n$. }

{\normalsize This is where we make use of the assumption that $\lambda_1 =
n\omega_1 p_1 \to \infty$ or $\lambda_2 = n\omega_2 p_2 \to \infty$.%
\footnote{%
If only $\lambda_3 \to \infty$, then replace the discussion of $\widetilde{%
s_2}$ with that of $\widetilde{s_3}$.} Since $n y_2(a) = -\lambda_1 -
\lambda_2 - \cdots$, this limiting assumption guarantees that the integral
exhibits exponential concentration. On the other hand, if the limit-suprema
of both $\lambda_1$ and $\lambda_2$ are both bounded above, then the
concentration is of polynomial type, and $\widetilde{s_2}$ is impossible to
analyze using Laplace's method. }

{\normalsize To proceed, we multiply the function $y_2$ by $n$ and
differentiate.\footnote{%
The derivatives of the $\Gamma $ function are the digamma function $\psi (x)=%
\frac{d}{dx}\log \Gamma (x)=\frac{\Gamma ^{\prime }(x)}{\Gamma (x)}$ and the
polygamma function $\psi ^{(m)}(x)=\frac{d^{m}}{dx^{m}}\log \Gamma (x)$.}
The first and second derivatives of $n y_{2}(a)$ with respect to $a$ are
equal to }

{\normalsize
\begin{equation*}
ny_{2}^{\prime }(a)=\log (n\omega_{1}p_{1})+\log (n\omega_{2}p_{2})-2\psi
(a+1)-\frac{1}{a+1}+\sum_{j\neq 1,2}\frac{1}{\sqrt{n\omega_{j}p_{j}}}\frac{%
\phi (\frac{a-n\omega_{j}p_{j}}{\sqrt{n\omega_{j}p_{j}}})}{\Phi (\frac{%
a-n\omega_{j}p_{j}}{\sqrt{n\omega_{j}p_{j}}})}
\end{equation*}
}

{\normalsize and }

{\normalsize
\begin{equation*}
ny_{2}^{\prime \prime }(a)=-2\psi ^{(2)}(a+1)+\frac{1}{\left( a+1\right) ^{2}%
}-\sum_{j\neq 1,2}\frac{1}{n\omega_{j}p_{j}}\left( \frac{(\frac{%
a-n\omega_{j}p_{j}}{\sqrt{n\omega_{j}p_{j}}})\phi (\frac{a-n\omega _{j}p_{j}%
}{\sqrt{n\omega_{j}p_{j}}})}{\Phi (\frac{a-n\omega_{j}p_{j}}{\sqrt{%
n\omega_{j}p_{j}}})}+\frac{\phi (\frac{a-n\omega_{j}p_{j}}{\sqrt{%
n\omega_{j}p_{j}}})^{2}}{\Phi (\frac{a-n\omega_{j}p_{j}}{\sqrt{n\omega
_{j}p_{j}}})^{2}}\right) <0.
\end{equation*}
}

{\normalsize Let $\alpha_2$ denote the unique solution to $ny_2^{\prime
}(\alpha_2) = 0$ (equivalently, $y_2^{\prime }(\alpha_2) = 0$). By expanding
the function $y_2$ in a Taylor series about the critical point $\alpha_2$,
we have}

{\normalsize
\begin{eqnarray*}
\widetilde{s_{2}} &=&\zeta \int_{0}^{\infty }e^{n(y_{2}(a))}da=\zeta
\int_{0}^{\infty }e^{n(y_{2}(\alpha_2)+(a-\alpha_2)y_{2}^{\prime }(\alpha_2)+%
\frac{1}{2}(a-\alpha_2)^{2}y_{2}^{\prime \prime }(\alpha_2)+\dots)}da \\
&\sim &\zeta \int_{0}^{\infty }e^{n(y_{2}(\alpha_2)+\frac{1}{2}%
(a-\alpha_2)^{2}y_{2}^{\prime \prime }(\alpha_2))}da = \zeta e^{n
y_{2}(\alpha_2)} \int_{0}^{\infty }e^{\frac{n}{2}(a-\alpha_2)^{2}y_{2}^{%
\prime \prime }(\alpha_2)}da,
\end{eqnarray*}%
since $y_2^{\prime }(\alpha_2) = 0$, and the first term does not depend on $%
a $. The approximation is justified since $y_{1}$ has a unique maximum at $%
a=\alpha_2$, and the terms decay rapidly in $n$ since $y^{\prime \prime
}(\alpha_2) <j 0$.\footnote{%
Importantly, the approximation is of type $\sim$ (instead of $\approx$)
since the two integral expressions agree in leading order terms (hence their
ratio converges to $1$). The difference between the two integrals grows
exponentially in $n$, but at a much lower rate, so it may be safely ignored.}
}

{\normalsize We may now exploit the central observation of Laplace's method:
the integrand is an unnormalized Gaussian density function (with mean $%
\alpha_2$ and variance $1 / ( ny_2^{\prime \prime }(\alpha_2))$), so the
integral over the whole real line may be computed explicitly. The
contribution over the negative real axis is negligible as $n \to \infty$
since $\alpha_2 > 0$, so
\begin{equation*}
s_2 \sim \widetilde{s_2} \sim \zeta e^{ny_{2}(\alpha_2)}
\int_{-\infty}^{\infty }e^{\frac{1}{2}(a-\alpha_2)^{2}y_{2}^{\prime \prime
}(\alpha_2)}da = \zeta e^{ny_{2}(\alpha_2)}\sqrt{\frac{2\pi }{%
n|y_{2}^{\prime \prime }(\alpha_2)|}}.
\end{equation*}
By a similar argument, we also have that $s_u \sim \widetilde{s_u} \sim
\zeta e^{ny_{u}(\alpha_u)}\sqrt{\frac{2\pi }{n|y_{u}^{\prime \prime
}(\alpha_u)|}}$ for $u=3,23$. Since $PP_{A,1} = s_2 + s_3 + s_{23}$, this
expression is dominated by the fastest-growing terms. }

{\normalsize Let $y_{max}(n) := \max_{u = 2, 3, 23} y_u(\alpha_u)$ denote
the maximum of the three exponential rates, which is realized by one, two or
all three of the indices. The expression $PP_{A,1}$ thus grows like $%
\widetilde{PP_{A,1}} := \zeta \rho(n) e^{n y_{max}(n)}$, where the
correction term $\rho(n)$ (defined in \eqref{rhodef}) involves only those
indices $u \in \{2,3,23\}$ for which $y_u(\alpha_u)$ attains the maximum.
This completes the proof.}
\end{proof}